\documentclass[a4paper,11pt]{amsart}

\numberwithin{equation}{section}

\usepackage{amsfonts}
\usepackage{amsthm}
\usepackage{amssymb}
\usepackage{amsmath}

\usepackage{amssymb,amsthm}
\usepackage{enumerate}

\usepackage[matrix,arrow,tips,curve]{xy}

\newenvironment{sis}{\left\{\begin{aligned}}{\end{aligned}\right.}

\theoremstyle{plain}
\newtheorem{thm}{Theorem}[section]
\newtheorem{lemma}[thm]{Lemma}
\newtheorem{prop}[thm]{Proposition}
\newtheorem{cor}[thm]{Corollary}

\newtheorem{fact}[thm]{Fact}

\newtheorem{lemdef}[thm]{Lemma-Definition}

\theoremstyle{definition}
\newtheorem{defi}[thm]{Definition}
\newtheorem{nota}[thm]{}
\newtheorem{question}[thm]{Question}

\theoremstyle{remark}

\newtheorem{remark}[thm]{Remark}

\usepackage{layout}
\usepackage{fullpage}

\newcommand{\la}{\longrightarrow}

\newcommand{\w}{\widetilde}
\newcommand{\wh}{\widehat}
\newcommand{\ov}{\overline}
\newcommand{\un}{\underline}

\newcommand{\val}{\operatorname{val}}

\renewcommand{\a}{\mathfrak{a}}

\def\X{\mathcal X}

\def\N{\mathcal N}

\def\O{\mathcal O}

\def\E{\mathcal E}
\def\P{\mathbb P}

\def\ner{\operatorname{N}(\Pic^0 \X_K)}
\def\nerd{{\operatorname{N}}(\Pic^d \X_K)}

\def\dcg{\Delta _X}

\def\LX{\Lambda _X}

\def\int{M_X}

\def\vectd{d_1,\ldots, d_{\gamma}}

\def\mc{\underline{c}_i}

\def\md{\underline{d}}

\newcommand{\Z}{\mathbb{Z}}
\newcommand{\Q}{\mathbb{Q}}

\newcommand{\Xsm}{X_{\rm sm}}

\newcommand{\Spec}{\operatorname{Spec}}

\newcommand{\im}{\operatorname{Im}}
\newcommand{\End}{\operatorname{End}}
\newcommand{\rk}{\operatorname{rk}}
\renewcommand{\Im}{\operatorname{Im}}

\newcommand{\mdeg}{\operatorname{{\underline{deg}}}}

\def\JPX{J_X^P(\un{q})}
\def\JPS{J_{X,S}^P(\un{q})}

\def\I{\mathcal I}

\newcommand{\Gr}{\operatorname{Gr}}
\renewcommand{\SS}{\mathfrak{S}}

\def\Picfun{\mathcal{P}ic}

\def\jacf{\Pic^{\underline 0}_f}

\newcommand{\Pic}{\operatorname{Pic}}

\newcommand{\pfun}[1]{{\Picfun}_f^{#1}}
\newcommand{\picf}[1]{\Pic_f^{#1}}

\newcommand{\picX}[1]{\Pic^{#1}X}

\newcommand{\Gm}{\mathbb{G}_m}

\newcommand{\hXS}{{\widehat{X}}_S}
\newcommand{\XSh}{\widehat{X_S}}

\newcommand{\hX}{{\widehat{X}}}

\newcommand{\Xsing}{X_{\rm{sing}}}
\newcommand{\Xsep}{X_{\rm{sep}}}

\newcommand{\NF}{\rm NF}

\newcommand{\Jsm}{J_f^{\sigma}(\un q)_{\rm sm}}

\newcommand{\JXsm}{J_X^P(\un q)_{\rm sm}}
\newcommand{\UXsm}{U_X(\un q)_{\rm sm}}

\newcommand{\Jprim}{J_{\hXS}^P(\qSh)_{\rm prim}}

\newcommand{\Bss}{B_{\Gamma\setminus S}(\un{q})}
\newcommand{\Bqs}{B_{\Gamma\setminus S}^{v_0}(\un{q})}

\newcommand{\qS}{\un{q^S}}
\newcommand{\qSh}{\widehat{\un{q^S}}}
\newcommand{\qh}{\widehat{q}}

\begin{document}


\title[Fine compactified Jacobians]{Fine compactified Jacobians}

\author[Melo]{Margarida Melo}
\address{Departamento de Matem\'atica, Universidade de Coimbra,
Largo D. Dinis, Apartado 3008, 3001 Coimbra (Portugal)}
\email{mmelo@mat.uc.pt}

\author[Viviani]{Filippo Viviani}
\address{
Dipartimento di Matematica,
Universit\`a Roma Tre,
Largo S. Leonardo Murialdo 1,
00146 Roma (Italy).
}
\email{viviani@mat.uniroma3.it}

\thanks{The first author was partially supported by the FCT project \textit{Espa\c cos de Moduli em Geometria Alg\'ebrica} (PTDC/MAT/111332/2009)
and by the FCT project \textit{Geometria Alg\'ebrica em Portugal} (PTDC/MAT/099275/2008).
The second author was supported by the grant FCT-Ci\^encia2008 from CMUC (University of Coimbra)
and by the FCT project \textit{Espa\c cos de Moduli em Geometria Alg\'ebrica} (PTDC/MAT/111332/2009).}



\begin{abstract}
We study Esteves's fine compactified Jacobians for  nodal curves.
We give a proof of the fact that, for a one-parameter regular local smoothing of a nodal curve $X$, the relative smooth locus of a relative fine compactified Jacobian
is isomorphic to the N\'eron model of the Jacobian of the general fiber, and thus it provides a modular compactification of it.
We show that each fine compactified Jacobian of $X$ admits a stratification in terms of certain
fine compactified Jacobians of partial normalizations of $X$ and, moreover, that it can be realized as a quotient of the smooth locus of a suitable
fine compactified Jacobian of the total blowup of $X$.
Finally, we determine when a fine compactified Jacobian is isomorphic to the corresponding Oda-Seshadri's coarse compactified Jacobian.
\end{abstract}

\maketitle

\tableofcontents

\section*{Introduction}

\begin{nota}{\it Motivation}

The Jacobian variety of a smooth curve is an abelian variety that carries important information about the curve itself. Its properties have been widely studied along the decades,
 giving rise to a significant amount of beautiful mathematics.

However, for singular (reduced) curves, the situation is more involved since the generalized Jacobian variety is not anymore an abelian variety, once it is, in general, not compact.
The problem of compactifying it is, of course, very natural, and it is considered to go back to the work of Igusa in \cite{igusa}  and Mayer-Mumford in \cite{MM}
in the 50's--60's.
Since then, several solutions appeared, differing from one another in various aspects as the generality of the construction, the modular description of the boundary and the functorial properties.

For families of irreducible curves, after the important work of D'Souza in \cite{dsouza}, a very satisfactory solution has been found by Altman and Kleiman in \cite{altman}:
their relative compactification is a fine moduli space, i.e. it admits a universal, or Poincar\'e, sheaf after an \'etale base change.

For reducible curves, the problem of compactifying the generalized Jacobian variety is much more intricate from a combinatorial and also functorial point of view.
The case of a single curve over an algebraically closed field was dealt with by Oda-Seshdari in \cite{OS} in the nodal case and
by Seshadri in \cite{ses} in the general case. For families of reducible curves, a relative compactification is provided by the work of Simpson in \cite{simpson}, which in great
generality deals with coherent sheaves on families of projective va\-rie\-ties.
A different approach is that of considering the universal Picard scheme over the moduli space of smooth curves and compactify it over the moduli space of stable curves. This point of view was the one
considered by Caporaso in \cite{caporaso} and by Pandharipande in \cite{pandha} (the later holds more generally for bundles of any rank) and by Jarvis in \cite{Jar}.
A common feature of these compactifications is that they are constructed using geometric invariant theory (GIT), hence they only give coarse moduli spaces for their corresponding
moduli functors. We refer to \cite{alexeev} and \cite{CMKV2} for an account on the way the different coarse compactified Jacobians for nodal curves relate to one another.


The problem of constructing fine compactified jacobians for reducible curves remained open until the work of Esteves in \cite{est1}.
Given a family $f:\X\to S$ of reduced curves endowed with a vector bundle $\E$ of integral slope, called polarization, and with a section $\sigma$, Esteves constructs
an algebraic space $J_{\E}^\sigma$ over $S$, which is a fine moduli space for
simple torsion-free sheaves on the family satisfying a certain stability condition with respect to $\E$ and $\sigma$ (see \ref{finecoarse}).
The algebraic space $J_{\E}^{\sigma}$ is always proper over $S$ and, in the case of a single curve $X$ defined over an algebraically closed field, it is indeed
a projective scheme (see \cite[Thm. 2.4]{est2}).

However, not much is known on  the geometry of Esteves's fine compactified Jacobians, for example how do they vary with the polarization and the choice of a section or how 
do they relate to the coarse compactified Jacobians. This last problem started to be investigated by Esteves in \cite{est2}, where  
a sufficient condition ensuring that a fine compactified Jacobian is isomorphic to the corresponding coarse compactified Jacobian (in the sense of \ref{finecoarse})
is found for curves with locally  planar singularities.

\end{nota}

\begin{nota}{\emph{Results}}

The aim of the present work is to study the geometry of Esteves's fine compactified Jacobians for a  nodal curve $X$ over an algebraically closed field $k$.
We introduce the notation $J_X^P(\un q)$ for the fine compactified Jacobians of $X$, where $P$ is a smooth point of $X$ and $\un q=\{\un q_{C_i}\}$ is a collection of rational numbers,
one for each irreducible component $C_i$ of $X$, summing up to an integer number $\displaystyle |\un q|:=\sum_{C_i} \un q_{C_i}\in \Z$ (which corresponds to the choice of a polarization, 
see \ref{finecoarse}).

Our first result is Theorem \ref{picner}, where we show that fine compactified Jacobians $J_X^P(\un q)$ provide a geometrically meaningful compactification of N\'eron models of nodal curves
or, according to the terminology  of \cite[Def. 2.3.5]{capsurvey} and  \cite[Def. 1.4 and Prop. 1.6]{capNtype}, that they  are of N\'eron-type. Explicitly, this means the following:
given a one-parameter  regular local smoothing $f:\X\to S=\Spec(R)$ of $X$ with a section $\sigma$ such that $P=\sigma(\Spec(k))$ (see \ref{notner}),
where $R$ is a Henselian DVR  with algebraically closed residue field $k$ and quotient field $K$, consider the
relative fine compactified Jacobian $J_f^{\sigma}(\un q)$, having special fiber isomorphic to $J_X^P(\un q)$ and general fiber isomorphic to $\Pic^{|\un q|}(\X_K)$.
Then the $S$-smooth locus of $J_f^{\sigma}(\un q)$, which consists of the sheaves on $\X$ whose restriction to $X=\X_k$ is locally free (see Fact \ref{smooth}),
is naturally isomorphic to the N\'eron model  $N(\Pic^{|\un q|}\X_K)$ of the degree $|\un q|$ Jacobian of the general fiber $\X_K$ of $f$.
In particular, one gets that, independently of the choice of the polarization $\un q$ and of the smooth point $P\in X_{\rm sm}$, the number of irreducible components of the fine
compactified Jacobians $J_X^P(\un q)$  is always equal to the complexity $c(\Gamma_X)$ of the dual graph $\Gamma_X$ of the curve $X$, or equivalently to the cardinality
of the degree class group $\Delta_X$ (see \ref{nota-group}). A different proof of this result already appears in the (unpublished) PhD thesis  of 
Busonero (\cite{Bus}).

Next, we show in Theorem \ref{strata-Jac} that the fine compactified Jacobians $J_X^P(\un q)$ of $X$ admit a canonical stratification
\[J_X^P(\un q)=\coprod_{\emptyset\subseteq S\subseteq X_{\rm sing}} J_{X,S}^P(\un q)\]
where $J_{X,S}^P(\un q)$ is the locally closed subset consisting of sheaves $\I\in J_X^P(\un q)$ that are not free exactly at $S\subseteq X_{\rm sing}$
and $J_{X,S}^P(\un q)$ is not empty if and only if the partial normalization $X_S$ of $X$ at $S$ is connected. We show that the closure of
$J_{X,S}^P(\un q)$ in $J^P_X(\un q)$ is equal to the union of the strata $J_{X,S'}^P(\un q)$ such that $S\subseteq S'$ and that it is
canonically isomorphic to a fine compactified Jacobian $J_{X_S}^P(\un{q^S})$ for a suitable polarization $\un{q^S}$ of $X_S$ (see \ref{pola-nota}). In particular,
each stratum $J_{X,S}^P(\un q)$ is a disjoint union of $c(\Gamma_{X_S})$ copies of the generalized Jacobian $J(X_S)$ of $X_S$. Combined with the previous result, 
this implies that fine compactified Jacobians of a nodal curve $X$ yield compactifications of N\'eron models of $X$ such that the boundary is made of N\'eron models of certain 
partial normalizations of $X$. 

In Theorem \ref{quot-blow}, we describe  $J_X^P(\un q)$ as a quotient of the smooth locus of a fine compactified Jacobian $J^P_{\hX}(\un{\qh})$ for a suitable
polarization $\un{\qh}$ on the total blowup $\hX$ of $X$ (see \ref{pola-nota}). In Theorem \ref{quot}, we show that a similar relation holds for the relative fine
compactified Jacobians of suitable one-parameter regular local smoothings of $X$ and $\hX$.
In particular, the fine compactified Jacobian $J_X^P(\un q)$ is a quotient
of the special fiber of the N\'eron model of $\hX$ in degree $|\un q|$.

Note that the above results were proved by Caporaso in \cite{capneron} and \cite{capNtype} for the coarse canonical degree-$d$
compactified Jacobians $\ov{P_X^d}$ (see \ref{nota-Jac}\eqref{nota-Jac3}) for a special class of stable curves $X$, called  $d$-general (see Remark \ref{can-gen}).
Our results can be seen as a generalization of her results to arbitrary nodal curves and to any polarization.

Finally, in Theorem \ref{nondeg-thm}, we determine for which polarizations $\un q$ and points $P\in X_{\rm sm}$, the natural map (see \ref{finecoarse})
\[\Phi : J_X^P(\un q)\longrightarrow U_X(\un q)\]
from Esteves's fine compactified Jacobians to the corresponding Oda-Sesha\-dri's coarse com\-pacti\-fied Jacobian is an isomorphism. In particular, we prove that this problem
depends only on $\un q$ and not on $P$ and that the sufficient conditions on $\un q$ found by Esteves in \cite{est2} are also necessary.




\end{nota}

\begin{nota}{\it Outline of the paper}

The paper is organized as follows.
In Section 1, we collect all the notations and basic properties about nodal curves and their combinatorial invariants (dual graph, degree class group, polarizations)
that we are going to use in the sequel. Moreover, we review the theory of N\'eron models
for Jacobians and the main properties of Esteves's fine compactified Jacobians as well as  Oda-Seshadri's, Seshadri's, Caporaso's and Simpson's coarse compactified Jacobians  for nodal curves. We also compare these constructions among each others and we establish formulae linking the different notations.

Section 2 is entirely devoted to the proof of a technical result in graph theory, that is a key ingredient for the results in the subsequent sections.


In Section 3 we prove that fine compactified Jacobians are of N\'eron type.


In Section 4 we describe a stratification of $J_X^P(\un q)$ in terms of fine compactifed Jacobians of partial normalizations of $X$.


Section 5
is devoted to show how to realize  fine compactified Jacobians of $X$ as quotients of the N\'eron model of the total blowup $\hX$ of $X$.

In Section 6 we characterize those polarizations for which Esteves's fine compactified Jacobians are isomorphic to Oda-Seshadri's  coarse compactified Jacobians.

\begin{nota}{\it Further questions and future work}

In the present paper we deal with nodal curves mainly because of the combinatorial tools that we use to prove our results, e.g. the dual graph associated to
a nodal curve. It is likely, however, that some of our results could be extended to more general singular curves, e.g. curves with locally planar singularities
(see \cite{AKI} for the relevance of locally planar singularities in the context of compactified Jacobians of singular curves).

The results of this paper show that the fine compactified Jacobians $J_X^P(\un q)$ of a nodal curve $X$  share very similar properties regardless of the polarization $\un q$
and the choice of the smooth point $P\in X_{\rm sm}$. The following question arises naturally
\begin{question}
For a given nodal curve $X$, how do the fine compactified Jacobians $J_X^P(\un q)$ change as the polarization $\un q$ and the smooth point $P\in X_{\rm sm}$ vary?
\end{question}

Note also that, by our comparison's result between fine compactified Jacobians and coarse compactified Jacobians (see Theorem \ref{nondeg-thm}), the above problem
is also closely related to the problem of studying the variation of GIT in the Oda-Seshadri's construction  of coarse compactified Jacobians of $X$.  In turn, this problem
seems to be related to wall-crossing phenomena for double Hurwitz numbers (see \cite{GJV} and \cite{CJM}). We plan to explore this fascinating connection in the future.

Recently, compactified Jacobians of integral curves have played an important role in the
celebrated proof of the Fundamental Lemma, since they appear naturally as fibers of the Hitchin's fibration in the case where the spectral curve is integral
(see \cite{Lau}, \cite{LN}, \cite{Ngo}). 
We plan to extend this description to nodal (reducible) spectral curves using fine compactified Jacobians. We expect that the results on the geometry of fine compactified Jacobians 
described here  can give important insights on the singularities of the fibers of the Hitchin map in the case where the spectral curve is reducible.

\vspace{0.5cm}

After this preprint was posted on arXiv, Jesse Kass posted the preprint \cite{Kas2} (based on his PhD thesis \cite{Kas}), where he extends our Theorem \ref{picner} to a larger class
of singular curves. Moreover,  he pointed out to us that our stratification of the fine compactified Jacobians of nodal curves (see Section \ref{Strat-sec}) is similar 
to the stratification by local type that the author describes in \cite[Sec. 5.3]{Kas}.


\end{nota}

\end{nota}

\paragraph { \bf Acknowledgements}

The present collaboration started  during our stay at IMPA, Rio de Janeiro, in July-August 2009. We would like to thank IMPA for the hospitality
and in particular Eduardo Esteves for the kind invitation and for sharing with us many enlightening ideas and suggestions. We would also like to thank Marco Pacini for many useful conversations,
Eduardo Esteves for some comments on an early draft of this manuscript and Jesse Kass for sending us a copy of his PhD thesis \cite{Kas}.

\section{Preliminaries and notations}
\label{not}

Throughout this paper, $R$ will be a Henselian (e.g. complete) discrete valuation ring (a DVR) with algebraically closed residue field
$k$ and quotient field  $K$.  We set $B=\Spec(R)$.

\begin{nota}{\emph{Nodal curves}}
\label{notnodal}

By a genus $g$ nodal curve $X$ we mean a projective and reduced curve of arithmetic genus $g:=1-\chi(\O_X)$ over $k$ having only
nodes as singularities. We will denote by     $\omega_X$ the canonical or dualizing sheaf of $X$.
We denote by     $\gamma_X$ (or simply $\gamma$) the number of irreducible components of $X$
and by $C_1,\ldots ,C_{\gamma}$ its irreducible components.

A \emph{subcurve} $Y\subset X$ is a closed subscheme of $X$ that is a curve, or in other words
$Y$ is the union of some irreducible components of $X$. We say that $Y$ is a proper subcurve, and we write $Y\subsetneq X$, if $Y$ is a subcurve of $X$
and $Y\neq X$. For any proper subcurve
$Y\subsetneq X$, we set $Y^c:=\ov{X\setminus Y}$ and we call it the complementary subcurve
of $Y$.
For a subcurve $Y\subset X$, we denote by     $g_Y$ its arithmetic genus and
by $\delta_Y:=|Y\cap Y^c|$ the number of nodes where $Y$  intersects the complementary curve $Y^c$. Then, the adjunction formula gives
\[
w_Y:=\deg (\omega_X)_{|Y}=2g_Y-2+\delta_Y.
\]
We denote by   $X_{\rm sm}$ the smooth locus of $X$ and by $\Xsing$ the set of nodes of $X$. We set $\delta=\delta_X:=|\Xsing|$.
The set of nodes $\Xsing$ admits a partition
\[\Xsing=X_{\rm ext}\coprod X_{\rm int},\]
 where $X_{\rm ext}$ is the subset of $\Xsing$ consisting of the nodes at which
two different irreducible components of $X$ meet (we call these external nodes),
and $X_{\rm int}$ is the subset of $\Xsing$ consisting of the nodes which are
self-intersection of  an irreducible component of $X$ (we call these internal
nodes).

We denote by     $\Gamma_X$ the \emph{dual graph} of $X$. With a slight abuse of notation,
we identify the edges $E(\Gamma_X)$ of $\Gamma_X$ with the nodes $\Xsing$ of $X$ and the vertices $V(\Gamma_X)$ of
$\Gamma_X$ with the irreducible components of $X$.
Note that the subcurves of $X$ correspond to the subsets of $V(\Gamma_X)$ via the following
bijection: we associate to a set of vertices $W\subseteq V(\Gamma_X)$ the subcurve $X[W]$ of $X$
given by the union of the irreducible components corresponding to the
vertices which belong to $W$.
Given a smooth point $P\in \Xsm$, we denote by $v_P$ the vertex corresponding
to the unique irreducible component of $X$ on which $P$ lies.

A node $N\in X_{\rm ext}$ is called a \emph{separating node} if $X-N$ is not connected.
Since $X$ is itself connected, $X-N$ would have two connected components.
Their closures are called the \emph{tails} attached to $N$. We denote by     $\Xsep\subset X_{\rm ext}$ the set of separating nodes of $X$.
Following \cite[Sec. 3.1]{est2}, we say that a subcurve $Y$ of $X$ is a \emph{spine} if $Y\cap Y^c\subset \Xsep$.
Note that the union of spines is again a spine and the connected components
of a spine are spines. A tail (attached to some separating node $N\in \Xsep$) is a spine
$Y$ such that $Y$ and $Y^c$ are connected and conversely.

Given a subset $S\subset \Xsing$, we denote by     $X_S$ the \emph{partial normalization} of $X$ at $S$ and by
$\widehat{X_S}$ the \emph{partial blowup} of $X$ at $S$, where (with a slight abuse of terminology) by blowup of $X$ at $S$ we mean the nodal curve $\widehat{X_S}$
obtained from $X_S$ by inserting a $\mathbb P^1$ attached at every pair of points of $X_S$ that are preimages of a node $n\in S$. We call such a $\P^1\subset \widehat{X_S}$ the exceptional 
component lying  above $n\in S$ and we denote by $E_S\subset \widehat{X_S}$ the union of all the exceptional components.
Note that we have a commutative diagram:
\begin{equation}\label{diag-S}
\xymatrix{X_S\ar@{^{(}->}[rr]^{i_S} \ar@{->>}[dr]_{\nu_S} & & \widehat{X_S}\ar@{->>}[dl]^{\pi_S} \\
& X &
}
\end{equation}
Here $\nu_S$ is the partial normalization map, $\pi_S$ contracts to $p\in S$ the exceptional component lying above $p$
and the inclusion $i_S$ realizes $X_S$ as the complementary subcurve of $E_S\subset \widehat{X_S}$.
We denote the total blowup of $X$ by  $\widehat{X}$
and the natural map to $X$ by  $\pi:\widehat{X}\to X$.

For a given subcurve $Y$  of $X$ denote by   $Y_S\subset X_S$ the preimage of $Y$ under $\nu_S$. Note that
$Y_S$ is the partial normalization of $Y$ at $S\cap Y$ and that every subcurve $Z\subset X_S$ is of the form $Y_S$ for some uniquely determined subcurve $Y\subset X$, namely $Y=\nu_S(Z)$.

The dual graph $\Gamma_{X_S}$ of $X_S$ is equal to the graph $\Gamma_X\setminus S$ obtained from $\Gamma_X$
by deleting all the edges belonging to $S$.
The dual graph $\Gamma_{\widehat{X_S}}$ of $\widehat{X_S}$ is equal to the graph $\widehat{(\Gamma_X)_S}$ obtained from
$\Gamma_X$ by adding a new vertex in the middle of every edge  belonging to $S$.

\end{nota}

\begin{nota}{\emph{Degree class group}}
\label{nota-group}

We call the elements $\md = (\vectd)$ of $\Z^{\gamma}$ \emph{multidegrees}.
We set  $|\md |:=\sum_{1}^{\gamma} d_i$ and call it the total degree of $\un d$.
For a line bundle $L\in \Pic X$ its multidegree  is
$ \mdeg L :=(\deg_{C_1}L,\ldots,\deg_{C_{\gamma}}L)$ and its (total) degree is
$\deg L:=\deg _{C_1}L+\ldots+\deg _{C_{\gamma}}L$.


Given $\md\in \Z^\gamma$ we set
$
\picX{\md}:=\{L\in \Pic X: \mdeg L = \md\}
$.
Note that $ \picX{\underline 0}:=\{L\in \Pic X: \mdeg L = (0,\ldots, 0)\}$ is a group (called the
{\it generalized Jacobian} of $X$ and denoted by $J(X)$) with respect to the tensor product of line bundles and each
$\Pic^{\md}(X)$ is a torsor under $\Pic^{\un 0}(X)$.
We set
$
\picX{d}:=\{L\in \Pic X:\deg L = d\}= \coprod_{|\md|=d}\picX{\md}.
$

For every component $C_i$ of $X$ denote
\[
\begin{matrix}\delta_{i,j}:= & \left\{\begin{array}{l}\;\;\, |C_i\cap C_j| \, \mbox{ if }i\not= j,\\  \\
-\delta_{C_i} \, \mbox{ if } i=j.\\
\end{array}\right.\end{matrix}
\]
For every $i=1,\ldots ,\gamma$ set
$
\mc :=(\delta_{1,i},\ldots ,\delta_{\gamma ,i})\in \Z ^{\gamma}.
$
Then  $|\mc|=0$ for all $i=1,\ldots, \gamma$
and the matrix $\int$ whose columns are the $\mc$ can be viewed as an {\it intersection matrix} for $X$.
Consider the sublattice $\LX $ of $\Z^{\gamma}$ of rank $\gamma-1$ spanned by the $\mc$
\[
\LX :=< \underline{c}_1, \ldots ,\underline{c}_{\gamma} >.
\]

\begin{defi}
We say that two multidegrees $\md$ and $\md'$ are equivalent, and write $\md\equiv \md'$,
if and only if $\md-\md'\in \LX$. The equivalence classes of multidegrees that sum up to $d$ are denoted by
\[
\dcg^d:= \{ \md \in \Z ^{\gamma}: |\md | =d\}/_\equiv.
\]
Note that $\dcg:=\dcg^0$ is a finite group and that  each $\dcg^d$ is a torsor
under $\dcg$.
The group $\Delta_X$ is
known in the literature under many different names (see \cite{BMS2} and the references therein); we will follow the terminology introduced in \cite{caporaso} and call it the \emph{degree class group} of $X$.
\end{defi}

We shall denote the elements in $\dcg^d$ by lowercase greek letters $\delta$ and write
$\md \in \delta$ meaning that the class $[\md]$ of $\md $ is $\delta$.

A well-known theorem in graph theory, namely Kirchhoff's Matrix Tree Theorem (see e.g. \cite[Thm. 1.6]{BMS} and the references therein), asserts that, if $X$ is connected, the cardinality of $\dcg$
(and hence of each $\dcg^d$) is equal to the
\emph{complexity} $c(\Gamma_X)$ of the dual graph $\Gamma_X$ of $X$, that is the number of spanning trees of $\Gamma_X$.
Note that $c(\Gamma_X)>0$ if and only if $X$ is  connected.

In the sequel, we will use the following result
which gives a formula for the complexity of $\widehat{\Gamma_S}$ (see the notation in \ref{notnodal}):

\begin{fact}\cite[Thm. 3.4]{BMS}\label{compl-blow}
For any $S\subset E(\Gamma)$, we have that
\[c(\widehat{\Gamma_S})=\sum_{\emptyset\subseteq S'\subseteq S} c(\Gamma\setminus S').\]
\end{fact}

\end{nota}

\begin{nota}\emph{N\'eron models of Jacobians}
\label{notner}

A \emph{one-parameter   regular local smoothing} of $X$ is a morphism $f:\X\to B$ where
$\X$ is a regular surface,
such that the special fiber  $\X_k$ is isomorphic to $X$ and the generic fiber $\X_K$ is a smooth curve.


Fix $f:\X\to B$ a one-parameter   regular local smoothing of $X$.
Let $\Picfun _f $ denote the \emph{relative Picard functor} of $f$ (often denoted
$ \Picfun _{\X/B}$ in the literature, see \cite[Chap. 8]{BLR} for the general theory).
$  \Picfun _f^d$ is the subfunctor of line bundles of relative degree $d$.
$\Picfun _f$ (resp. $\pfun{d}$) is represented by a scheme $\Pic_f$ (resp. $\Pic^d_f$) over $B$, see
\cite[Thm. 8.2]{BLR}. Note that $\Pic_f$ and $\Pic_f^d$ are not separated over $B$ if $X$ is reducible.

For each multidegree $\un d\in \Z^{\gamma}$, there exists a separated closed subscheme $\picf{\md}\subset \Pic _f^d$  parametrizing
line bundles of relative degree $d$ whose restriction to the closed fiber has multidegree $\md$.
In other words, the special fiber of $\picf{\md}$ is isomorphic to $\Pic^{\md}(X)$ while, clearly, the general fiber
is isomorphic to $\Pic^d(\X_K)$.
Note that $\jacf$ is a group scheme over $B$ and that the $\picf{\md}$'s are torsors under $\jacf$.
It is well-known (see \cite[Sec. 3.9]{capneron}) that if $\un d\equiv \un d'$ then
there is a canonical isomorphism (depending only on $f$)
\[
\iota_f(\md,\md'):\Pic_f^{\md}\la \Pic_f^{\md'}
\]
which restricts to the identity on the generic fiber.
The isomorphism $\iota_f(\un d,\un d')$ is given by tensoring with a line bundle on $\X$ of the form $\O_{\X}(\sum_i n_i C_i)$, for suitably
chosen integers  $n_i\in \Z$ such that $\sum_i n_i=0$.
We shall therefore  identify $\Pic _f^{\md}$ with $\Pic _f^{\md'}$ for all pairs of equivalent multidegrees
$\md$ and
$\md '$. Thus for every $\delta \in \dcg^d$ we define
\begin{equation}
\label{picdel}
\picf{\delta}:=\picf{\md}\  \
\end{equation}
 for every $\md \in \delta$.

For any integer $d$, denote by   $\nerd$ the \emph{N\'eron model} over $B$ of the degree-$d$ Picard variety $\Pic^d\X_K$
of the generic fiber $\X_K$.  Recall that $\nerd$ is smooth and separated over $B$, the generic fiber
$\nerd_K$ is isomorphic to $\Pic^d \X_K$ and $\nerd$ is uniquely characterized by the
following universal property (the N\'eron mapping property, cf.  \cite[Def. 1]{BLR}):
every $K$-morphism $u_K:Z_K\la \nerd_K=\Pic^d \X_K$ defined on  the generic fiber of some scheme
$Z$ smooth over $B$ admits a unique extension  to a $B$-morphism $u:Z\la \nerd$.
Moreover, $\ner$ is a $B$-group scheme while, for every $d\in \Z$, $\nerd$ is a torsor under $\ner$.

The N\'eron models $\nerd$ can be described as the biggest separated quotient of $\Pic_f^d$ (\cite[Sec. 4.8]{raynaud}).
Indeed, since $\Pic_f^d$ is smooth over $B$ and its general fiber is isomorphic to $\Pic^d(\X_K)$, the N\'eron mapping
property yields a map
\begin{equation}\label{sepquot}
 q:\Pic_f^d\to \nerd.
 \end{equation}
The scheme $\Pic_f^d$ can be described as
\[\Pic_f^d \cong \frac{\coprod _{\un d\in \Z^{\gamma}\: :\: |\un d|=d}\picf{\un d}}{\sim_K},
\]
where $\sim_K$ denotes the gluing of the schemes $\picf{\un d}$ along their general fibers, which are isomorphic
to $\Pic^d(\X_K)$. On the other hand, the N\'eron model $\nerd$ can be explicitly described as follows
\begin{fact}\cite[Lemma 3.10]{capneron}\label{expl-Neron}
We have a canonical $B$-isomorphism
\begin{equation}\label{neronglue}
\nerd \cong \frac{\coprod _{\delta \in \dcg^d}\picf{\delta}}{\sim_K}.
\end{equation}
\end{fact}
Therefore, the above map $q$ sends each $\picf{\un d}$ isomorphically into $\picf{[\un d]}$ and identifies $\picf{\un d}$ with
$\picf{\un{d'}}$ if and only if $\un d\equiv \un{d'}$.

Note that, from Fact \ref{expl-Neron}, it follows that the special fiber of the N\'eron model $\nerd$, which we will denote by
$N_X^d$,  is isomorphic to a disjoint union of $c(\Gamma_X)$'s copies of the generalized Jacobian $J(X)$ of $X$.


\end{nota}




\begin{nota}\label{pola-nota}
\emph{Polarizations}

\begin{defi}\label{pola-def}
A \emph{polarization} on $X$ is a $\gamma$-tuple of rational numbers $\un q=\{\un q_{C_i}\}$, one for each irreducible component $C_i$ of $X$,
such that $|\un q|:=\sum_i \un q_{C_i}\in \Z$.
\end{defi}

Given a subcurve $Y \subset X$, we set $\un{q}_Y:=\sum_j \un{q}_{C_j}$ where the sum runs
over all the irreducible components $C_j$ of $Y$. Note that giving a polarization $\un q$ is the same as giving an
assignment $(Y\subset X)\mapsto \un q_Y$ which is additive on $Y$, i.e. such that if $Y_1,Y_2\subset X$ are two subcurves of $X$ without
common irreducible components then $\un q_{Y_1\cup Y_2}=\un q_{Y_1}+\un q_{Y_2}$ and such that $\un q_X\in \Z$.

If $Y\subset X$ is a subcurve of $X$ such that
$\un q_Y-\displaystyle \frac{\delta_Y}{2}\in \Z$, then we define the \emph{restriction} of the polarization $\un q$
to $Y$ as the polarization $\un q_{|Y}$ on $Y$ such that
\begin{equation}\label{res-pol}
(\un q_{|Y})_{Z}=\un q_Z-\frac{|Z\cap Y^c|}{2},
\end{equation}
for any subcurve $Z\subset Y$.

Given a subset $S\subset \Xsing$ and a polarization $\un q$ on $X$, we define a polarization $\un{q^S}$
(resp. $\widehat{\un{q^S}}$) on the partial normalization $X_S$ (resp. the partial blowup $\widehat{X_S}$)
of $X$ at $S$ (see the notation in \ref{notnodal}).

\begin{lemdef}\label{pola-norma}
The formula
\begin{equation*}
\un{q^S}_{Y_S}:= \un{q}_Y-\frac{|S_e^Y|}{2}-|S_i^Y|,
\end{equation*}
for any subcurve $Y_S\subset X_S$, where $S_e^Y:=S\cap Y\cap Y^c$ and $S_i^Y:=S\cap (Y\setminus Y^c)$, defines a polarization on $X_S$.
\end{lemdef}
\begin{proof}
We have to show that $\un{q^S}$  is additive, i.e. that for any two subcurves
$Y_S$ and $Z_S$ of $X_S$ without common components it holds $\un{q^S}_{Y_S\cup Z_S}=\un{q^S}_{Y_S}+\un{q^S}_{Z_S}$.
This follows from the additivity of $\un q$ and the
easily checked formulas:
\begin{equation}\label{equa-S}
\begin{sis}
& |S_i^{Y\cup Z}|=|S_i^{Y}|+|S_i^{Z}|+|S\cap Y\cap Z|,\\
& |S_e^{Y\cup Z}|=|S_e^{Y}|+|S_e^{Z}|-2|S\cap Y\cap Z|.\\
\end{sis}
\end{equation}
We conclude by observing that $\un{q^S}_{X_S}=\un q_X-|S|\in \Z$.
\end{proof}

The proof of the following Lemma-Definition is trivial.

\begin{lemdef}\label{pola-blow}
The formula
\begin{equation*}
\widehat{\un{q^S}}_{Z}=\begin{cases}
0 & \text{ if } Z\subseteq E_S, \\
\un{q}_{\pi_S(Z)} & \text{ if } Z\not\subseteq E_S,
\end{cases}
\end{equation*}
for any subcurve $Z\subset \widehat{X_S}$,  defines a polarization on $\widehat{X_S}$.
\end{lemdef}

\noindent In the special case of the total blowup $\widehat{X}=\widehat{X_{\Xsing}}$, we set
$\widehat{\un q}:=\widehat{\un{q^{\Xsing}}}$.

In the last part of the paper, we will need the concept of generic and non-degenerate polarizations.
First, imitating \cite[Def. 3.4]{est2}, we give the following

\begin{defi}\label{def-int}
A polarization $\un q$ is called \emph{integral} at a subcurve $Y\subset X$ if
$\un q_Z-\displaystyle \frac{\delta_Z}{2} \in \Z$ for any connected component $Z$ of $Y$ and of $Y^c$.
\end{defi}

Using the above definition, we can give the following

\begin{defi}\label{def-pol}
\noindent
\begin{enumerate}[(i)]
\item \label{def-pol1} A polarization $\un q$ is called \emph{general} if it is not integral at any proper subcurve $Y\subsetneq X$.
\item \label{def-pol2} A polarization $\un q$ is called \emph{non-degenerate} if it is not integral at any proper subcurve $Y\subsetneq X$
which is not a spine of $X$.
\end{enumerate}
\end{defi}

\end{nota}

\begin{nota}\emph{Semistable, torsion-free, rank $1$ sheaves}
\label{notsheaves}

Let $X$ be a connected nodal curve of genus $g$.
Let $\I$ be a coherent sheaf on $X$. We say that $\I$ is \emph{torsion-free} (or \emph{depth $1$} or \emph{of pure dimension} or \emph{admissible})
if its associated points are generic points of $X$. Clearly, a torsion-free sheaf $\I$ can be not 
free only at the nodes of $X$;
we denote by   $NF(\I)\subset \Xsing$ the subset of the nodes
of $X$ where $\I$ is not free ($\NF$ stands for not free).
We say that $\I$ is of \emph{rank 1} if $\I$ is invertible on a
dense open subset of $X$. We say that $\I$ is \emph{simple} if $\End(\I) = k $. Each line
bundle on $X$ is torsion-free of rank $1$ and simple.

For each subcurve $Y$ of $X$, let $\I_Y$ be the restriction $\I_{|Y}$ of $\I$ to $Y$ modulo torsion.
If $\I$ is a torsion-free (resp. rank $1$)
sheaf on $X$, so is $\I_Y$ on $Y$.
We let $\deg_Y (\I)$ denote the degree of $\I_Y$, that is, $\deg_Y(\I) := \chi(\I_Y )-\chi(\O_Y)$.

It is a well-known result of Seshadri (see \cite{ses}) that torsion-free, rank 1 sheaves on $X$ can be described either via line bundles on partial normalizations of $X$ or via certain line bundles on partial blowups of $X$.
The precise statement is the following

\begin{prop}\label{sheaf-linebun}
\noindent
\begin{enumerate}[(i)]
\item \label{sheaf-linebun1}
For any $S\subset \Xsing$, the commutative diagram \eqref{diag-S} induces a commutative diagram
\begin{equation}\label{diag-Pic}
\xymatrix{
\Pic(X_S)\ar[dr]_{(\nu_S)_*}^{\cong} & & \Pic(\widehat{X_S})_{\rm prim}\ar@{->>}[dl]^{(\pi_S)_*}
\ar@{->>}[ll]_{i_S^*} \\
& {\rm Tors}_S(X) & \\
}
\end{equation}
where $\Pic(\widehat{X_S})_{\rm prim}$ denotes the line bundles on $\widehat{X_S}$ that have degree
$-1$ on each exceptional component of the morphism $\pi_S$ and ${\rm Tors}_S(X)$ denotes the set
of torsion-free, rank $1$ sheaves $\I$ on $X$ such that $\NF(\I)=S$.
Moreover we have that
\begin{enumerate}
\item  The maps $i_S^*$ and $(\pi_S)_*$ are surjective;
\item The map $(\nu_S)_*$ is bijective with inverse given by sending
a sheaf $\I\in {\rm Tors}_S(X)$ to the line bundle on $X_S$ obtained as the quotient of
$(\nu_S)^*(\I)$  by its torsion subsheaf.
\end{enumerate}


\item \label{sheaf-linebun2}
The above diagram \eqref{diag-Pic} is equivariant with respect to the natural actions of the generalized
Jacobians of $X_S$, $\XSh$ and $X$ and the natural morphisms:
\begin{equation}\label{diag-Jac}
\xymatrix{
J(X_S) & & J(\XSh) \ar@{->>}[ll]_{i_S^*} \\
& J(X) \ar@{->>}[ul]^{\nu_S^*} \ar[ur]_{\pi_S^*}^{\cong}& \\
}
\end{equation}
Explicitly, for any $L\in \Pic(\widehat{X_S})_{\rm prim}$, $M\in \Pic(X_S)$, $\alpha\in J(X)$ and
$\beta\in J(\XSh)$, we have that
\begin{equation}\label{form-equiv}
\begin{sis}
& i_S^*(\beta\otimes L)=i_S^*(\beta)\otimes i_S^*(L), \\
& (\pi_S)_*(\pi_S^*\alpha\otimes L)=\alpha\otimes (\pi_S)_*(L), \\
& (\nu_S)_*(\nu_S^*\alpha \otimes M)=\alpha\otimes (\nu_S)_*(M).
\end{sis}
\end{equation}
In particular, the action of $J(X)$ on ${\rm Tors}_S(X)$
factors through the map $\nu_S^*:J(X)\twoheadrightarrow J(X_S)$.
 \item \label{sheaf-linebun3}
For any subcurve $Y\subset X$ and any $M\in \Pic(X_S)$, it holds
\[\deg_Y(\nu_S)_*(M)=\deg_{Y_S}M+|S_i^Y|,\]
where $S_i^Y:=S\cap (Y\setminus Y^c)$ (as in Lemma-Definition \ref{pola-norma}).
 \end{enumerate}
\end{prop}
\begin{proof}
Part (i) is a reformulation of \cite[Lemma 1.5(i) and Lemma 1.9]{alexeev}).

Part (ii) follows from the multiplicativity of pull-back map $i_S^*$ and the projection formula
applied to the morphisms $\nu_S$ and $\pi_S$.

Part (iii): First of all observe that the restriction $((\nu_S)_*M)_{|Y}$ is equal to the pushforward via $(\nu_S)_{|Y_S}:Y_S\to Y$ of the restriction $M_{|Y_S}=M_{Y_S}$.
Since $(\nu_S)_{|Y_S}$ is a finite map, we get the equality $\chi(((\nu_S)_*M)_{|Y})=\chi(M_{Y_S})$ which, combined with Riemann-Roch, gives that
\begin{equation*}
\deg ((\nu_S)_*M)_{|Y}+1-g(Y)=\chi(((\nu_S)_*M)_{|Y})=\chi(M_{Y_S})=\deg_{Y_S}M+1-g(Y_S).\tag{*}
\end{equation*}
Since $Y_S$ is the normalization of $Y$ at $S\cap Y$, we have that $g(Y_S)=g(Y)-|S\cap Y|$ which, combined with (*), gives that
\begin{equation*}
\deg ((\nu_S)_*M)_{|Y}=\deg_{Y_S}M+|S\cap Y|.\tag{**}
\end{equation*}
Clearly, the torsion subsheaf of $((\nu_S)_*M)_{|Y}$ is equal to $\displaystyle \bigoplus_{n \in S\cap Y\cap Y^c} \un{k}_n$, where  $\un{k}_n$ is
the skyscraper sheaf supported on $n$ and with stalk equal to the base field $k$. Therefore
\begin{equation*}
\deg_Y ((\nu_S)_*M)=\deg ((\nu_S)_*M)_{Y}=\deg ((\nu_S)_*M)_{|Y}-|S\cap Y\cap Y^c|. \tag{***}
\end{equation*}
We conclude by putting together (**) and (***).
\end{proof}

Later, we will need the concepts of semistability, $P$-quasistability and stability of a torsion-free, rank $1$ sheaf on $X$ with respect to a polarization on $X$
and to a smooth point $P\in \Xsm$.
Here are the relevant definitions.

\begin{defi}\label{sheaf-ss-qs}
\noindent Let $\un q$ be a polarization on $X$ and let $P\in\Xsm$ be a smooth point of $X$.
Let $\I$ be a torsion-free, rank-1 sheaf on $X$ of degree $d=|\un q|$.
\begin{enumerate}[(i)]
\item \label{sheaf-ss} We say that $\I$ is \emph{semistable} with respect to $\un q$ (or $\un q$-semistable) if
for every proper subcurve $Y$ of $X$, we have that
\begin{equation}\label{multdeg-sh1}
\deg_Y(\I)\geq \un q_Y-\frac{\delta_Y}{2}
\end{equation}
\item \label{sheaf-qs} We say that $\I$ is
$P$-\emph{quasistable} with respect to $\un q$ (or $\un q$-P-quasistable) if it is semistable with respect to $\un q$
and if the inequality (\ref{multdeg-sh1}) above is strict when $P\in Y$.
\item \label{sheaf-s} We say that $\I$ is \emph{stable} with respect to $\un q$ (or $\un q$-stable) if it is semistable with respect to $\un q$
and if the inequality (\ref{multdeg-sh1}) is always strict.
\end{enumerate}
\end{defi}

In what follows we compare our notation with the other notations used in the literature.


\begin{remark}\label{inequ}
\begin{enumerate}[(i)]
\noindent \item \label{inequ1} Given a vector bundle $E$ on $X$, we define the polarization $\un q^{E}$ on $X$ by setting
\[\un{q}^E_{Y}=-\frac{\deg(E_{|Y})}{\rk(E)}+\frac{\deg_{Y}(\omega_X)}{2},\]
 for each subcurve $Y$ (or equivalently
for each irreducible component $C_i$) of $X$.
Then it is easily checked that the above notions of semistability (resp. $P$-quasistability, resp. stability) with respect to $\un q^E$ agree with the
notions of semistability (resp. $P$-quasistability, resp. stability) with respect to $E$ in the sense of
\cite[Sec. 1.2]{est1}. Note that, for any subcurve $Y\subset X$ such that $\un q_Y-\displaystyle \frac{\delta_Y}{2}\in \Z$, we
have that $(\un q^E)_{|Y}=\un q^{E_{|Y}}$.
\item \label{inequ2} In the particular case where
\begin{equation}\label{can-pol}
\un{q}_Y=d\cdot \frac{\deg_Y(\omega_X)}{2g-2},
\end{equation}
for a certain integer $d\in \Z$, the inequality (\ref{multdeg-sh1}) reduces to the well-known basic inequality of
Gieseker-Caporaso (see \cite{caporaso}). In this case, $\un q$ will be called the \emph{canonical polarization of degree d}.
\end{enumerate}
\end{remark}

Given a  sheaf $\I$ semistable with respect to a polarization $\un q$, there are
connected subcurves $Y_1,\ldots,Y_q$ covering $X$ and a filtration
\[0 = \I_0 \subsetneqq \I_1 \subsetneqq \ldots, \subsetneqq \I_{q-1}\subsetneqq \I_q=\I\]
such that the quotient $\I_j/\I_{j+1}$ is a stable sheaf on $Y_j$ with respect to $\un q_{|Y_j}$ for
each $j = 1, \ldots, q$. The above filtration is called a \emph{Jordan-H\"older filtration}. The
sheaf $\I$ may have many Jordan-H\"older filtrations but the collection of subcurves
$\SS(\I) := \{Y_1,\ldots,Y_q\}$ and the isomorphism class of the sheaf
\[\Gr(\I) := \I_1/\I_0 \oplus \I_2/\I_1\oplus \ldots \oplus \I_q/\I_{q-1}\]
depend only on $\I$, by the Jordan-H\"older theorem. Notice that $\Gr(\I)$ is also $\un q$-semistable
and that
\[\Gr(\I)\cong \bigoplus_{Z\in \SS(\I)} \Gr(\I)_Z.\]
A $\un q$-semistable sheaf $\I$ is called \emph{polystable} if $\I\cong \Gr(\I)$.

We say that two $\un q$-semistable sheaves $\I$ and $\I'$ on $X$ are \emph{$S$-equivalent} if $\SS(\I) = \SS(\I')$
and $\Gr(\I) \cong \Gr(\I')$.  Note that in each $S$-equivalence class
of $\un q$-semistable sheaves, there is exactly one $\un q$-polystable sheaf.

\end{nota}




\begin{nota}\label{finecoarse}
{\emph{Fine and coarse compactified Jacobians}}

For any smooth point $P\in X$ and polarization $\un q$ on $X$, there is a $k$-projective variety $J_X^P(\un q)$,
which we call \emph{fine compactified Jacobian},
parametrizing  $\un q$-P-quasistable sheaves on the curve $X$
(see \cite[Thm. A, p. 3047]{est1} and \cite[Thm. 2.4]{est2}).
More precisely, $J_X^P(\un q)$ represents
the functor that associates to each scheme $T$ the set of
$T$-flat coherent sheaves $\I$ on $X\times T$ such that
$\I|_{X\times t}$ is $\un q$-P-quasistable
for each $t\in T$, modulo the following equivalence relation $\sim$.
We say that two such sheaves $\I_1$ and $\I_2$
are equivalent, and denote $\I_1\sim\I_2$,
if there is an invertible sheaf $\N$ on $T$ such that
$\I_1\cong\I_2\otimes p_2^*\N$, where $p_2: X\times T\to T$ is the
projection map.

There are other two varieties closely related to $J_X^P(\un q)$ (see \cite[Sec. 4]{est1}): the variety $J_X^s(\un q)$ parametrizing $\un q$-stable
sheaves and the variety $J_X^{ss}(\un q)$ parametrizing $\un q$-semi\-stable simple sheaves. We have open inclusions
\[J_X^s(\un q)\subset J_X^P(\un q)\subset J_X^{\rm ss}(\un q),\]
where the last inclusion follows from the fact that a $\un q$-P-quasistable sheaf is simple, as it follows
easily from \cite[Prop. 1]{est1}. It turns out that $J_X^s(\un q)$ is separated but, in general, not universally closed, while
$J_X^{ss}(\un q)$ is universally closed but, in general, not separated (see \cite[Thm. A]{est1}).

According to \cite[Thm. 15, p. 155]{ses}, there exists
a projective variety $U_X(\un q)$, which we call \emph{coarse compactified Jacobian},
coarsely representing the functor $\text{\bf U}$
that associates to each scheme $T$ the set of $T$-flat coherent sheaves $\I$ on $X\times T$ such that
$\I_{X\times t}$ is $\un q$-semistable for each $t\in T$.
More precisely, there is a map $\text{\bf U}\to U_X(\un q)$ such that, for any other $k$-scheme $Z$, each map $\text{\bf U}\to Z$
is induced by composition with a unique map $U_X(\un q)\to Z$.
Moreover, the $k$-points on $U_X(\un q)$ are in one-to-one correspondence
with the $S$-equivalence classes of $\un q$-semistable sheaves on $X$, or equivalently
with $\un q$-polystable sheaves on $X$ since in each $S$-equivalence class of $\un q$-semistable sheaves there exists exactly one $\un q$-polystable
sheaf. By convention, when we write $\I\in U_X(\un q)$, we implicitly assume that
$\I$ is polystable. We denote by
\[U_X^s(\un q)\subset U_X(\un q)\]
the open subset parametrizing $\un q$-stable sheaves.

Since $J_X^P(\un q)$ represents a functor,
there exists a universal $\un q$-P-quasistable sheaf on
$X\times J_X^P(\un q)$ (uniquely determined up to tensoring with the pull-back
of a line bundle on $J_X^P(\un q)$), and hence a well-defined induced map
\begin{equation}\label{Smap}
\Phi : J_X^P(\un q)\longrightarrow U_X(\un q).
\end{equation}
This map is surjective (by \cite[Thm. 7]{est1}) and its fibers parametrize
$S$-equivalence classes of $\un q$-P-quasistable sheaves (see also \cite[p. 178]{est2}).
The map $\Phi$ fits in the following diagram
\begin{equation}\label{Smap-dia}
\xymatrix{
J_X^s(\un q) \ar@{^{(}->}[r]\ar[d]^{\Phi^s}_{\cong} & J_X^P(\un q) \ar@{->>}[d]^{\Phi}\ar@{^{(}->}[r] & J_X^{ss}(\un q) \ar@{->>}[dl]^{\Phi^{ss}} \\
U_X^s(\un q) \ar@{^{(}->}[r] & U_X(\un q) &
}
\end{equation}

To compare our notations with the others used in the literature, we observe the following

\begin{remark}\label{nota-Jac}
\noindent
\begin{enumerate}[(i)]
\item \label{nota-Jac1} Given a vector bundle $E$ on $X$ and a smooth point $P\in \Xsm$,  the variety
$J_X^P(\un q^E)$ coincides with the variety $J_E^P$ in Esteves's notation (see \cite{est1}). Similarly,
the variety $J_X^s(\un q^E)$ (resp. $J_X^{ss}(\un q^E)$) coincides with $J_E^s$ (resp. $J_E^{ss}$)
in Esteves's notation.

\item \label{nota-Jac1bis}
Let $\phi$ be an element of $\partial C_1(\Gamma_X,\Q)\subset C_0(\Gamma_X,\Q)$ (see \cite[Sec. 1]{alexeev}), i.e.
a collection of rational numbers $\{\phi_v\}$ for any vertex $v$ of $\Gamma_X$ such that $\displaystyle \sum_{v\in V(\Gamma_X)}\phi_v=0$.
We can associate to $\phi$ a polarization $\un \phi$ such that $|\un \phi|=0$ by putting
\begin{equation}\label{pola-homology}
\un \phi_{C_v}=\phi_v
\end{equation}
if $C_v$ is the irreducible component of  $X$ corresponding to the vertex $v$ of $\Gamma_X$. Then the  Oda-Seshadri's compactified Jacobian ${\rm Jac}(X)_{\phi}$ is isomorphic to $U_X(\un \phi)$
(see \cite{OS} and \cite{alexeev}).

 Conversely, given a polarization $\un q$, consider a polarization $\un d$ such that  $|\un q|=|\un d|$ and such that $\un d$ is integral,
i.e. $\un d_Y\in \Z$ for any subcurve $Y\subseteq X$.
Define a new polarization $\un \phi$ by $\un \phi_Y:=\un q_Y-\un d_Y$ for any subcurve $Y\subseteq X$.  In particular, we have that
$|\un \phi|=0$. Define an element $\phi\in \partial C_1(\Gamma_X,\Q)\subset C_0(\Gamma_X,\Q)$ by the equation (\ref{pola-homology}).
Then the variety $U_X(\un q)$ is isomorphic to ${\rm Jac}(X)_{\phi}$.
Note that this is independent of the choice of the auxiliary integral polarization $\un d$ because
we have an isomorphism ${\rm Jac}(X)_{\phi}\cong {\rm Jac}(X)_{\phi+\psi}$ for any $\psi\in \partial C_1(\Gamma_X,\Z)\subset C_0(\Gamma_X,\Z)$.

\item \label{nota-Jac2}
Given a pair $(\a, \chi)$, where $\chi\in \Z$ and $\a=\{\a_{C_i}\}$ is a polarization such that $|\a|=1$,
consider the polarization $\un q$ defined by
\[\un q_Y=\a_Y\chi+\frac{\deg_Y(\omega_X)}{2},\]
for every subcurve $Y\subset X$. Then the variety $U_X(\un q)$ coincides with the variety
$U_X(\a, \chi)$ in Seshadri's notation (see \cite{ses}).

\item\label{nota-Jac2bis}
Given an ample line bundle $L$ on $X$ and an integer $d\in \Z$, consider the polarization $\un q$ defined by
\[\un q_Y=\frac{\deg_Y(\omega_X)}{2}+\frac{\deg_Y(L)}{\deg(L)}(d-g+1), \]
for every subcurve $Y\subseteq X$.
Then the  Simpson's moduli space (see \cite{simpson}) ${\rm Jac}(X)_{d,L}$
of $S$-equivalence classes of torsion-free, rank one sheaves of degree $d$ that are slope-semistable with respect to $L$ is isomorphic to $U_X(\un q)$
(see \cite{alexeev}). However, note that, contrary to what asserted in \cite[Sec. 2.1]{alexeev}, 
it is not true that every $U_X(\un q)$ with $|\un q|=d$ is isomorphic to ${\rm Jac}(X)_{d,L}$ for some ample line bundle $L$ on $X$. 
For instance, if $d=g-1$ then it follows easily from the above equation that all the Simpson's compactified Jacobians ${\rm Jac}(X)_{g-1, L}$ are isomorphic among them 
regardless of the chosen $L$ (as observed also in  \cite[Lemma 3.1]{alexeev}), while there many compactified Jacobians of the form $U_X(\un q)$ with $|\un q|=g-1$, 
just as in every other degree $d$! 

\item \label{nota-Jac3} In the particular case where $\un q$ is the canonical polarization of degree $d$ (see \ref{inequ}(\ref{inequ2})), the variety
$U_X(\un q)$ coincides with the variety $\ov{P_X^d}$ in Caporaso's notation (see \cite{caporaso}) and it will be called the \emph{coarse canonical}
degree $d$ compactified Jacobian of $X$. Moreover, we  set $J_X^{d,P}:=J_X^P(\un{q})$ and call it the
\emph{fine canonical} degree $d$ compactified Jacobian of $X$ with respect to $P$. This notation
agrees with the one introduced in \cite[Sec. 2.4]{CCE}. In particular, we have a surjective map
$J_X^{d,P}\twoheadrightarrow \ov{P_X^d}.$
\end{enumerate}
\end{remark}

In what follows, we will need the following  results concerning the smooth loci of $J_X^P(\un q)$ (or $J_X^s(\un q)$ or $J_X^{ss}(\un q)$)
and $U_X(\un q)$.


\begin{fact}\label{smooth}
\noindent
\begin{enumerate}[(i)]
\item \label{smooth1} The variety $J_X^P(\un q)$ (resp. $ J_X^s(\un q)$, resp. $ J_X^{ss}(\un q)$) is smooth at $\I$ if and only if $\I$
is a line bundle on $X$.
\item \label{smooth2} The variety $U_X(\un q)$ is smooth at a polystable sheaf $\I$  if and only if $\I$ is locally free at all non-separating nodes of $X$.
\end{enumerate}
\end{fact}
For the proof of part \eqref{smooth1}, observe that, since $J_X^P(X)$ is a fine compactified Jacobian, the completion of the local ring of $J_X^P(\un q)$ at  $\I$ is isomorphic
to the miniversal deformation ring of $\I$. The same thing is true for $ J_X^s(\un q)$, resp. $ J_X^{ss}(\un q)$. The result then follows from \cite[Lemma 3.14]{CMKV2}.
Part \eqref{smooth2} follows from \cite[Thm. B(ii)]{CMKV2}.

\vspace{0,2cm}

Now fix a one-parameter   regular local smoothing $f:\X\to B=\Spec(R)$ of $X$ (see \ref{notner}).

It follows from \cite{Ish} that there exists a $B$-scheme $U_f(\un q)$ whose special fiber is isomorphic to
$U_X(\un q)$ and whose general fiber is isomorphic to  $\Pic^{|\un q|}(\X_K)$.
Denote by $U_f^s(\un q)$ the open subset of $U_f(\un q)$ whose special fiber is isomorphic to $U_X^s(\un q)\subset U_X(\un q)$ 
and whose general fiber is isomorphic to  $\Pic^{|\un q|}(\X_K)$.

Note that, since $R$ is assumed to be Henselian, for any $P\in X_{\rm sm}$ there exists a section $\sigma: B \to \X$ of $f$ such that $\sigma(\Spec k)=P$ (see e.g. \cite
[Prop. 14]{BLR}).
Conversely, every section $\sigma$ of $f$ is such that  $\sigma(\Spec k)$ is a smooth point of $\X_k=X$ (see e.g. \cite[Chap. 9, Cor. 1.32]{Liu}).
Fix now a section $\sigma$ of $f$ and let  $P:=\sigma(\Spec k)\in X_{\rm sm}$. 
Then, according to \cite[Thm. A and Thm. B]{est1}, there exist $B$-schemes $J_f^s(\un q)$, $J_f^{\sigma}(\un q)$ and $J_f^{ss}(\un q)$ together with open inclusions
\[J_f^s(\un q)\subset J_f^{\sigma}(\un q)\subset J_f^{ss}(\un q),\]
such that the general fibers over $B$  of the above schemes is $\Pic^{|\un q|}(\X_K)$ while the special fibers
are isomorphic to, respectively,  $J_X^s(\un q)$, $J_X^P(\un q)$ and $J_X^{ss}(\un q)$.
The above diagram (\ref{Smap-dia}) becomes the special fiber of the following diagram of $B$-schemes
\begin{equation}\label{Smap-fam}
\xymatrix{
J_f^s(\un q) \ar@{^{(}->}[r]\ar[d]^{\Phi_f^s}_{\cong} & J_f^{\sigma}(\un q) \ar@{->>}[d]^{\Phi_f}\ar@{^{(}->}[r] & J_f^{ss}(\un q) \ar@{->>}[dl]^{\Phi_f^{ss}} \\
U_f^s(\un q) \ar@{^{(}->}[r] & U_f(\un q) &
}
\end{equation}

\end{nota}

\section{Graph-theoretic results}

\begin{nota}{\emph{Notations.}}
Let $\Gamma$ be a finite graph with vertex set $V(\Gamma)$ and edge set $E(\Gamma)$.
We allow loops or multiple edges, although, in what follows, loops will play no role,
i.e. we could consider the graph $\w{\Gamma}$ obtained from $\Gamma$ by removing all the loops
and obtain exactly the same answers we get for $\Gamma$.

We will be interested in two kinds of \emph{subgraphs} of $\Gamma$:
\begin{itemize}
\item Given a subset $T\subset E(\Gamma)$, we denote by   $\Gamma\setminus T$ the
subgraph of $\Gamma$ obtained from $\Gamma$ by deleting the edges belonging to $T$.
Thus we have that $V(\Gamma\setminus T)=V(\Gamma)$ and $E(\Gamma\setminus T)=E(\Gamma)\setminus T$.
The subgraphs of the form $\Gamma\setminus T$ are called complete subgraphs.
\item Given a subset $W\subset V(\Gamma)$, we denote by   $\Gamma[W]$ the subgraph
whose vertex set is $W$ and whose edges are all the edges of $\Gamma$ that join two vertices in $W$.
The subgraphs of the form $\Gamma[W]$ are called induced subgraphs and we say that
$\Gamma[W]$ is induced from $W$.
\end{itemize}
If $W_1$ and $W_2$ are two disjoint subsets of $V(\Gamma)$, then we set
$\val(W_1,W_2):=$ $|E(\Gamma[W_1],$ $\Gamma[W_2])|$,
where $E(\Gamma[W_1], \Gamma[W_2])$ is the subset of $E(\Gamma)$ consisting of all the edges
of $\Gamma$ that join some vertex of $W_1$ with some vertex of $W_2$. We call
$\val(W_1,W_2)$ the \emph{valence} of the pair $(W_1,W_2)$.
For a subset $W\subset V(\Gamma)$, we denote by   $W^c:=V(\Gamma)\setminus W$ its
complementary subset. We set $\val(W)=\val(W^c):=\val(W,W^c)$ and call it the
valence of $W$. In particular
$\val(\emptyset)=\val(V(\Gamma))=0$. Note that for $w\in V(\Gamma)$, the valence
$\val(w)$ is the number of edges joining $w$ with a vertex of $\Gamma$ different from $w$
i.e. loops are not taken into account
in our definition of valence.

Given a subset $S\subseteq E(\Gamma)$, we define the valence of the pair $(W_1,W_2)$
of disjoint subsets $W_1,W_2\subset V(\Gamma)$ with respect to $S$ to be $\val_S(W_1,W_2):=|S\cap E(\Gamma[W_1],\Gamma[W_2])|$.
Obviously, we always have that $\val_S(W_1,W_2)\leq \val(W_1,W_2)$ with equality if $S=E(\Gamma)$.

Note that the valence is additive: if $W_1, W_2, W_3$ are pairwise disjoint subsets
of $V(\Gamma)$, we have that
\begin{equation}\label{add-val}
\val(W_1\cup W_2,W_3)=\val(W_1,W_3)+\val(W_2,W_3).
\end{equation}
A similar property holds for $\val_S$.
\end{nota}

\begin{nota}{\emph{$0$-cochains.}}
Given an abelian group $A$ (usually $A=\Z, \Q$), we define the space $C^0(\Gamma,A)$
of $0$-cochains with values in $A$ as the free $A$-module $A^{V(\Gamma)}$ of functions from
$V(\Gamma)$ to $A$. If $\un{d}\in C^0(\Gamma,A)$, we set
\[\begin{sis}
&\un{d}_v:= \un{d}(v)\in A \text{ for any } v\in V(\Gamma),\\
&\un{d}_W:=\sum_{w\in W} \un{d}_w \in A \text{ for any } W\subseteq V(\Gamma),\\
&|\un{d}|:=\un{d}_{V(\Gamma)}\in A.
\end{sis}\]
For any element $a\in A$, we set
\[C^0(\Gamma,A)_a:=\{\un{d}\in C^0(\Gamma,A)\: :\: |\un{d}|=a\}\subseteq
C^0(\Gamma,A).
\]


Given a subset $W\subset V(\Gamma)$, we will denote by   $\un{\chi(W)}\in C^0(\Gamma,\Z)$ the
characteristic function of $W$, i.e. the element of $C^0(\Gamma,\Z)$ uniquely
defined by
\begin{equation}\label{char-fun}
\un{\chi(W)}_v=
\begin{cases}
1 & \text{ if } v\in W,\\
0 & \text{otherwise}.
\end{cases}
\end{equation}

The space of $0$-cochains with values in $A$ is endowed with an endomorphism,
called Laplacian and denoted by $\Delta_0$ (see for example \cite[Pag. 169]{BdlHN}),
defined as
\begin{equation}\label{Lapla}
\Delta_0(\un{d})_v:=-\un{d}_v \val(v) +\sum_{w\neq v} \un{d}_w \val(v,w).
\end{equation}
It is easy to check that $\Im(\Delta_0)\subset C^0(\Gamma,A)_0$. In the case where $A=\Z$ and
$\Gamma$ is connected, the kernel $\ker(\Delta_0)$ consists of the constant $0$-cochains
and therefore the quotient
\[\Pic(\Gamma):=\frac{C^0(\Gamma,\Z)_0}{\im(\Delta_0)}\]
is a finite group, called the Jacobian group (see \cite{BdlHN}).

For any $d\in \Z$, the set $C^0(\Gamma,\Z)_d$ is clearly a torsor for the group $C^0(\Gamma,\Z)_0$.
Therefore, the subgroup $\im(\Delta_0)$ acts on the
sets $C^0(\Gamma,\Z)_d$ and
\begin{equation}\label{lat-class0}
|\Pic(\Gamma)|=\left|\frac{C^0(\Gamma,\Z)_d}{\im(\Delta_0)}\right|.
\end{equation}

\begin{remark}\label{dcg}
Let $X$ be a connected nodal curve and consider the dual graph of $X$, $\Gamma_X$.
Then $\Gamma_X$ is connected and it is easy to check that
$\Pic(\Gamma_X)\cong\dcg$ (see \ref{nota-group}).
Moreover, for any $d\in \Z$, there is a bijection $\frac{C^0(\Gamma_X,\Z)_d}{\im(\Delta_0)}\leftrightarrow \dcg^d$. In particular,  we have that:
\begin{equation}\label{lat-class}
c(\Gamma_X)=\left|\frac{C^0(\Gamma_X,\Z)_d}{\im(\Delta_0)}\right|.
\end{equation}
\end{remark}

For later use, we record the following formula
(for any $W, V\subseteq V(\Gamma)$):
\[\Delta_0(\un{\chi(V)})_W=\sum_{w\in W}\left[-\un{\chi(V)}_w \val(w) +
\sum_{v\neq w}\un{\chi(V)}_v \val(v,w)\right]=\]
\[=\sum_{w\in V\cap W}\left[-\val(w) + \sum_{w\neq v\in V} \val(v,w)\right]+
\sum_{w\in W\setminus V}\sum_{v\in V} \val(v,w)=\]
\[=\sum_{w\in V\cap W}\left[-\sum_{v\in V^c} \val(v,w)\right]+
\sum_{w\in W\setminus V}\sum_{v\in V} \val(v,w)\]
\[=-\val(V\cap W, V^c)+\val(W\setminus V, V)=\]
\[=-\val(V\cap W, W\setminus V)-\val (V\cap W, (V\cup W)^c)+\val(W\setminus V, V\cap W)+\val(W\setminus V,V\setminus W )= \]
\begin{equation}\label{for-cha}
=-\val (V\cap W, (V\cup W)^c)+\val(W\setminus V,V\setminus W ).
\end{equation}

\end{nota}

\begin{nota}{\emph{Quasistable $0$-cochains $\Bqs$}}

Throughout this subsection, we fix the following data:
\begin{enumerate}
\item A finite graph $\Gamma$;
\item $v_0\in V(\Gamma)$;
\item $S\subset E(\Gamma)$;
\item $\un{q}\in C^0(\Gamma,\Q)$ such that $q:=|\un{q}|\in \Z$.
\end{enumerate}
Since we will be using two different graphs throughout this section, $\Gamma$ and $\Gamma\setminus S$,
we will adopt the following convention on the notation used.
Given two disjoint subsets $W_1,W_2\subseteq V(\Gamma)=V(\Gamma\setminus S)$, we will be considering three
different notions of valence, namely:
\[\begin{sis}
& \val(W_1,W_2):=|E(\Gamma[W_1],\Gamma[W_2])|,\\
& \val_S(W_1,W_2):=|S\cap E(\Gamma[W_1],\Gamma[W_2])|,\\
& \val_{\Gamma\setminus S}(W_1,W_2):=|E((\Gamma\setminus S)[W_1],(\Gamma\setminus S)[W_2])|.
\end{sis}\]
Note that $\val(W_1,W_2)=\val_S(W_1,W_2)+\val_{\Gamma\setminus S}(W_1,W_2)$. As usual,
we set $\val(W):=\val(W,W^c)$ and similarly for $\val_S$ and $\val_{\Gamma\setminus S}$.

We now introduce the main characters of this subsection.

\begin{defi}\label{def-ss-qs}\noindent
\begin{enumerate}[(i)]
\item A $0$-cochain $\un{d}\in C^0(\Gamma,\Z)$  is
said to be semistable on $\Gamma\setminus S$ with respect to $\un{q}$ if the following two
conditions are satisfied:
\begin{enumerate}\label{ss-cond}
\item\label{ss-cond1} $|\un{d}|=q-|S|$;
\item\label{ss-cond2} $\un{d}_W+|S\cap E(\Gamma[W])| \geq \un{q}_W -\frac{\val(W)}{2}$ for any proper subset $W\subset V(\Gamma)$.
\end{enumerate}
We denote the set of all such $0$-cochains by
$\Bss$.
\item \label{qs-cond} A $0$-cochain $\un{d}\in C^0(\Gamma,\Z)$ is
said to be $v_0$-quasistable on $\Gamma\setminus S$ with respect to $\un{q}$
if $\un{d}\in \Bss$ and the inequality in (\ref{ss-cond2}) above is strict when $v_0\in W$.
We denote the set of all such $0$-cochains by $\Bqs$.
\end{enumerate}
\end{defi}

\begin{remark}\label{sym-inequ}
Let $\un{d}\in \Bss$ and $W$ a proper subset of $V(\Gamma)$.
By applying the condition (\ref{ss-cond2}) of Definition \ref{def-ss-qs} to $W^c\subset V(\Gamma)$ and using
(\ref{ss-cond1}), we get that
\begin{equation*}
\un{q}_W-\frac{\val(W)}{2}+\val_{\Gamma\setminus S}(W)=\un{q}_W+\frac{\val(W)}{2}-\val_S(W)
\geq \un{d}_W+|S\cap E(\Gamma[W])|.
\end{equation*}
If moreover $\un{d}\in \Bqs$ then the above inequality is strict if $v_0\not\in W$.
\end{remark}

We want to determine the cardinality of the set $\Bqs$. We begin with the following necessary condition
in order that $\Bqs$ is not empty. Later (see Corollary~\ref{cor-card}),
we will see that it is also a sufficient  condition.

\begin{lemma}\label{empty}
If $\Bqs\neq\emptyset$ then $\Gamma\setminus S$ is connected.
\end{lemma}
\begin{proof}
By contradiction, assume that $\Gamma\setminus S$ is not connected and $\Bqs\neq \emptyset$.
This means that there exist $\un{d}\in \Bqs$ and
a proper subset $W\subset V(\Gamma)$ such that $\val_{\Gamma\setminus S}(W)=0$.
By the Definition \ref{def-ss-qs} and Remark \ref{sym-inequ}, we get that
\[\un{q}_W -\frac{\val(W)}{2}\leq \un{d}_W+|S\cap E(\Gamma(W))|\leq \un{q}_W -\frac{\val(W)}{2} +
\val_{\Gamma\setminus S}(W)= \un{q}_W -\frac{\val(W)}{2}.
\]
This contradicts the fact that one of the above two inequalities
must be strict, according to whether $v_0\in W$ or $v_0\in W^c$.
\end{proof}

In what follows, we are going to consider the $0$-cochains $C^0(\Gamma\setminus S,\Z)$
endowed with
the Laplacian operator $\Delta_0$ as in (\ref{Lapla}) with respect to $\Gamma \setminus S$.
Note that, although $C^0(\Gamma\setminus S,\Z)=C^0(\Gamma,\Z)$
is independent of the chosen $S\subset E(\Gamma)$,
the Laplacian $\Delta_0$ depends on $S$.

\begin{prop}\label{card-qs}
If $\Gamma \setminus S$ is connected, then the composed map
\[\pi:\Bqs\subseteq C^0(\Gamma\setminus S,\Z)_{q-|S|} \twoheadrightarrow
\frac{C^0(\Gamma\setminus S,\Z)_{q-|S|}}{\im(\Delta_0)}.\]
is bijective.
\end{prop}
\begin{proof}
Consider the auxiliary map
\[\ov{\pi}:\Bss\subseteq C^0(\Gamma\setminus S,\Z)_{q-|S|} \twoheadrightarrow
\frac{C^0(\Gamma\setminus S,\Z)_{q-|S|}}{\im(\Delta_0)}.\]
Clearly we have that $\pi=\ov{\pi}_{|\Bqs}$.
We divide the proof in three steps.

\noindent \underline{STEP I: $\pi$ is injective.}

By contradiction, assume that there exist $\un{d}\neq \un{e}\in \Bqs$
such that $\pi(\un{d})=\pi(\un{e})$. This is equivalent to the existence
of an element $\un{t}\in C^0(\Gamma\setminus S, \Z)$ such that
$\Delta_0(\un{t})=\un{d}-\un{e}$.
Since $\un{d}, \un{e}\in \Bqs$, by Definition \ref{def-ss-qs} and Remark \ref{sym-inequ},
we get that for any proper subset $W\subset V(\Gamma)$:
\[\un{d}_W-\un{e}_W < \left(\un{q}_W+\frac{\val(W)}{2}-\val_S(W)\right)-
\left(\un{q}_W-\frac{\val(W)}{2}\right)=\]
\begin{equation}\label{eq1}
=\val(W)-\val_S(W)=\val_{\Gamma\setminus S}(W),
\end{equation}
where the inequality is strict since either $v_0\in W$ or $v_0\in W^c$.

Consider now the (non-empty) subset
\begin{equation*}
V_0:=\{v\in V(\Gamma)=V(\Gamma\setminus S)\: :\: \un{t}_v=\min_{w\in V(\Gamma)}
\un{t}_w:=l\}\subseteq V(\Gamma)=V(\Gamma\setminus S).
\end{equation*}
If $V_0=V(\Gamma\setminus S)$ then $\un{t}$ is a constant $0$-cochain
in $\Gamma\setminus S$, and therefore $0=\Delta_0(\un{t})=\un{d}-\un{e}$,
which contradicts the hypothesis that $\un{d}\neq \un{e}$. Therefore
$V_0$ is a proper subset of $V(\Gamma\setminus S)$.

From the definition (\ref{Lapla}), using the additivity of $\val_{\Gamma\setminus S}$
and the fact that $\un{t}_v\geq l$ for any $v\in V(\Gamma\setminus S)$ with equality
if $v\in V_0$, we get
\[\Delta_0(\un{t})_{V_0}=\sum_{v\in V_0}\left[-l \cdot \val_{\Gamma\setminus S}(v)+
\sum_{w\neq v} \un{t}_w \val_{\Gamma\setminus S}(v,w)\right]=\]
\[=\sum_{v\in V_0}\left[-l \cdot \val_{\Gamma\setminus S}(v)+
\sum_{w\in V_0\setminus\{v\}} l \cdot \val_{\Gamma\setminus S}(v,w)+
\sum_{w\in V_0^c} \un{t}_w \val_{\Gamma\setminus S}(v,w)\right]=\]
\[=\sum_{v\in V_0}\left[-l \cdot \val_{\Gamma\setminus S}(v)+
l \cdot \val_{\Gamma\setminus S}(v,V_0\setminus \{v\})+
\sum_{w\in V_0^c} \un{t}_w \val_{\Gamma\setminus S}(v,w)\right]=\]
\[=\sum_{v\in V_0}\left[-l \cdot \val_{\Gamma\setminus S}(v, V_0^c)+
\sum_{w\in V_0^c} \un{t}_w \val_{\Gamma\setminus S}(v,w)\right]=\]
\[=\sum_{v\in V_0, w\in V_0^c} (\un{t}_w-l) \val_{\Gamma\setminus S}(v,w)
\geq \]
\begin{equation}\label{eq2}
\geq \sum_{v\in V_0, w\in V_0^c} \val_{\Gamma\setminus S}(v,w)=
\val_{\Gamma\setminus S}(V_0,V_0^c)=\val_{\Gamma\setminus S}(V_0).
\end{equation}
Using the fact that $\Delta_0(\un{t})=\un{d}-\un{e}$,
 the above inequality (\ref{eq2}) contradicts the strict inequality (\ref{eq1})
for $W=V_0$, which holds since $V_0$ is a proper subset of $V(\Gamma\setminus S)$.

\underline{STEP II: $\ov{\pi}$ is surjective.}

We introduce two rational numbers measuring how far is an element $\un{d}\in C^0(\Gamma\setminus S,\Z)_{q-|S|}$
from being in $\Bss$.
For any $\un{d}\in C^0(\Gamma\setminus S,\Z)_{q-|S|}$ and any $W\subseteq V(\Gamma)$ (non necessarily proper), set
\begin{equation}
\begin{sis}
& \epsilon(\un{d},W):=\un{d}_W+|S\cap E(\Gamma[W])|-\un{q}_W-\frac{\val(W)}{2}+\val_S(W),\\
& \eta(\un{d},W):=-\un{d}_W-|S\cap E(\Gamma[W])|+\un{q}_W-\frac{\val(W)}{2}.
\end{sis}
\end{equation}
Using the two relations
\[\begin{sis}
& \un{d}_W+\un{d}_{W^c}+|S|=\un{q}_W+\un{q}_{W^c},\\
& |S|=|S\cap E(\Gamma[W])|+|S\cap E(\Gamma[W^c])|+\val_S(W),\\
\end{sis}\]
it is easy to check that
\begin{equation}\label{sym-inv}
\epsilon(\un{d},W)=\eta(\un{d},W^c).
\end{equation}
We set also for any $\un{d}\in C^0(\Gamma\setminus S,\Z)_{q-|S|}$
\begin{equation}
\begin{sis}
& \epsilon(\un{d}):=\max_{W\subseteq V(\Gamma)} \epsilon(\un{d},W),\\
& \eta(\un{d}):=\max_{W\subseteq V(\Gamma)} \eta(\un{d},W).\\
\end{sis}
\end{equation}
From equation (\ref{sym-inv}), we get that
\begin{equation}\label{equ-inv}
\epsilon(\un{d})=\eta(\un{d}).
\end{equation}


We will often use in what follows that the invariants $\epsilon$ and $\eta$ satisfy the following
additive formula:  for any disjoint subsets $W_1, W_2\subset V(\Gamma)$,
we have that
\begin{equation}\label{add-for}
\begin{sis}
& \epsilon(\un{d},W_1\cup W_2)=\epsilon(\un{d},W_1)+\epsilon(\un{d},W_2)+\val_{\Gamma\setminus S}(W_1,W_2), \\
& \eta(\un{d},W_1\cup W_2)=\eta(\un{d},W_1)+\eta(\un{d},W_2)+\val_{\Gamma\setminus S}(W_1,W_2).
\end{sis}
\end{equation}
Let us prove the second additive formula; the proof of the first one is similar and left to the
reader. Using the additivity (\ref{add-val}) of $\val$ and $\val_S$, we compute:
\[\eta(\un{d},W_1\cup W_2)=-\un{d}_{W_1\cup W_2}-|S\cap E(\Gamma[W_1\cup W_2])|+\un{q}_{W_1\cup W_2}
-\frac{\val(W_1\cup W_2)}{2}=\]
\[=-\un{d}_{W_1}-\un{d}_{W_2}-|S\cap E(\Gamma[W_1])|-|S\cap E(\Gamma[W_2])|-\val_S(W_1,W_2)
+\un{q}_{W_1}+\un{q}_{W_2}+\]
\[-\frac{\val{W_1}+\val{W_2}-2\val(W_1,W_2)}{2}= \eta(\un{d},W_1)+\eta(\un{d},W_2)-\val_S(W_1,W_2)+\]
\[+\val(W_1,W_2)=\eta(\un{d},W_1)+\eta(\un{d},W_2)+\val_{\Gamma\setminus S}(W_1,W_2).\]

\noindent For an element $\un{d}\in C^0(\Gamma\setminus S,\Z)_{q-|S|}$, consider the following sets:
\begin{equation*}
\begin{sis}
& S^+_{\un{d}}:=\{W\subseteq V(\Gamma)\: :\: \epsilon(\un{d},W)=\epsilon(\un{d})\},\\
& S^-_{\un{d}}:=\{W\subseteq V(\Gamma)\: :\: \eta(\un{d},W)=\eta(\un{d})\}.
\end{sis}
\end{equation*}
From formula (\ref{sym-inv}) and the equality $\epsilon(\un{d})=\eta(\un{d})$, it follows easily
that
\begin{equation}\label{comp-ext}
W\in S^+_{\un{d}} \Leftrightarrow W^c\in S^-_{\un{d}}.
\end{equation}
The sets $S^{\pm}_{\un{d}}$ are stable under intersection:
\begin{equation}\label{stab-inter}
W_1, W_2 \in S^{\pm}_{\un{d}} \Rightarrow W_1\cap W_2\in S^{\pm}_{\un{d}}.
\end{equation}
We will prove this for $S_{\un{d}}^+$; the proof for $S_{\un{d}}^-$ works
exactly the same. Let $\Pi_1:=W_1\setminus (W_1\cap W_2)$. Using the additivity
formula (\ref{add-for}) applied to the pair $(W_2, \Pi_1)$ of disjoint subsets of
$V(\Gamma)$ and the fact that $W_2\in S^+_{\un{d}}$, we get that
\begin{equation*}
0=\epsilon(\un{d})-\epsilon(\un{d},W_2)\geq \epsilon(\un{d},\Pi_1\cup W_2)-\epsilon(\un{d},W_2)= \epsilon(\un{d}, \Pi_1)+\val_{\Gamma\setminus S}(\Pi_1,W_2).
\end{equation*}
Using this inequality, the additivity formula (\ref{add-for}) for the disjoint pair $(W_1\cap W_2, \Pi_1)$ of
subsets of $V(\Gamma)$ and the fact that $W_1\in S_{\un{d}}^+$, we get that
\[\epsilon(\un{d})=\epsilon(\un{d}, W_1)=\epsilon(\un{d}, (W_1\cap W_2)\cup \Pi_1)=\]
\[=\epsilon(\un{d},W_1\cap W_2)
+\epsilon(\un{d},\Pi_1)+\val_{\Gamma\setminus S}(\Pi_1,W_1\cap W_2)\]
\[ \leq  \epsilon(\un{d},W_1\cap W_2)+\epsilon(\un{d},\Pi_1)+\val_{\Gamma\setminus S}(\Pi_1,W_2)
\leq \epsilon(\un{d},W_1\cap W_2).\]
 By the maximality of $\epsilon(\un{d})$, we
conclude that $\epsilon(\un{d})=\epsilon(\un{d},W_1\cap W_2)$, i.e. that $W_1\cap W_2 \in S_{\un{d}}^+$.

Since the sets $S_{\un{d}}^{\pm}$ are stable under intersection, they admit minimum elements:
\begin{equation}
\Omega^{\pm}(\un{d}):=\bigcap_{W\in S_{\un{d}}^{\pm}} W\subseteq V(\Gamma).
\end{equation}
Note that (\ref{comp-ext}) implies that $\Omega^+(\un{d})^c\in S_{\un{d}}^-$.  Since $\Omega^-(\un{d})$
is the minimum element of $S_{\un{d}}^-$, we get that $\Omega^-(\un{d})\subseteq \Omega^+(\un{d})^c$,
or in other words
\begin{equation}
\Omega^+(\un{d})\cap \Omega^-(\un{d})=\emptyset.
\end{equation}
We set
\begin{equation*}
\Omega^0(\un{d}):=V(\Gamma)\setminus (\Omega^+(\un{d})\cup \Omega^-(\un{d})),
~\end{equation*}
so that $V(\Gamma)$ is the disjoint union of $\Omega^+(\un{d})$, $\Omega^-(\un{d})$ and
$\Omega^0(\un{d})$.

From (\ref{equ-inv}) and the fact that $\epsilon(\un{d},V(\Gamma))=
\eta(\un{d},V(\Gamma))=\epsilon(\un{d},\emptyset)=\eta(\un{d},\emptyset)=0$,
we get  that $\epsilon(\un{d})=\eta(\un{d})\geq 0$.
From \ref{def-ss-qs}(\ref{ss-cond}) and the definition of $\Omega^{\pm}(\un{d})$,
it follows that
\begin{equation}\label{proper-Omega}
\un{d}\in \Bss \Leftrightarrow \epsilon(\un{d}) \text{ or } \eta(\un{d})=0 \Leftrightarrow
\Omega^{+}(\un{d}) \text{ or } \Omega^-(\un{d})=\emptyset.
\end{equation}
Fix now an element $\un{d}\in C^0(\Gamma\setminus S,\Z)_{q-|S|}$ such that $\un{d}\not \in \Bss$.
Set
\begin{equation}\label{new-el}
\un{e}:=\un{d}+\Delta_0(\un{\chi(\Omega^+(\un{d}))}).
\end{equation}

\un{Claim:} The $0$-cochain $\un{e}$ satisfies one of the two following properties:
\begin{enumerate}[(i)]
\item $\epsilon(\un{e})<\epsilon(\un{d})$,
\item $\epsilon(\un{e})=\epsilon(\un{d})$ and
$\Omega^+(\un{e})\supsetneq \Omega^+(\un{d})$.
\end{enumerate}

Note that the Claim concludes the proof of Step II. Indeed, if $\un{e}$ satisfies condition (ii),
we can iterate the substitution (\ref{new-el}) until we reach an element $\un{e'}$
which satisfies condition (i), i.e. $\epsilon(\un{e'})<\epsilon(\un{d})$,
and such that $\un{e'}-\un{d}\in \im \Delta_0$.
Now observe that, if we set $N$ to be equal to two times the least common multiple
of all the denominators of the rational numbers $\{\un{q}_v\}_{v\in V(\Gamma)}$, then
$N\cdot \epsilon(\un{f})\in \Z,$
for any $\un{f}\in C^0(\Gamma\setminus S,\Z)$. Therefore, by iterating the substitution (\ref{new-el}),
we will finally reach an element $\un{e''}$ such that $\epsilon(\un{e''})=0$, i.e.
$\un{e''}\in \Bss$, and such that $\un{e'}-\un{d}\in \im \Delta_0$. This proves
that $\ov{\pi}$ is surjective.

Let us now prove the Claim. Take any subset $W\subset V(\Gamma)$ and decompose it as
a disjoint union
\begin{equation*}
W=W^+\coprod W^-\coprod W^0,
\end{equation*}
where $W^{\pm}=W\cap \Omega^{\pm}(\un{d})$ and $W^0=W\cap \Omega^0(\un{d})$.
Note that
\begin{equation}\label{ine1}
\epsilon(\un{d},W^+)\leq \epsilon(\un{d}),
\end{equation}
with equality if and only if $W^+=\Omega^+(\un{d})$ because of the minimality property
of $\Omega^+(\un{d})$. Applying (\ref{add-for}) to the disjoint pair $(\Omega^+(\un{d}), W^0)$,
we get
\[\epsilon(\un{d}, W^0)=\epsilon(\un{d},W^0\cup \Omega^+(\un{d}))-\epsilon(\un{d}, \Omega^+(\un{d}))
-\val_{\Gamma \setminus S}(W^0,\Omega^+(\un{d}))\leq \]
\begin{equation}\label{ine2}
\leq -\val_{\Gamma \setminus S}(W^0,\Omega^+(\un{d})),
\end{equation}
where we used that $\epsilon(\un{d}, W^0\cup \Omega^+(\un{d}))\leq
\epsilon(\un{d})=\epsilon(\un{d}, \Omega^+(\un{d}))$. Applying once more formula
(\ref{add-for}) to the disjoint pair $(W^-, \Omega^+(\un{d})\cup\Omega^0(\un{d}))$, we get
\[\epsilon(\un{d}, W^-)=\epsilon(\un{d},W^-\cup \Omega^+(\un{d})\cup\Omega^0(\un{d}))-
\epsilon(\un{d}, \Omega^+(\un{d})\cup\Omega^0(\un{d}))-\]
\begin{equation}\label{ine3}
-\val_{\Gamma \setminus S}(W^-,\Omega^+(\un{d})\cup\Omega^0(\un{d}))
\leq -\val_{\Gamma \setminus S}(W^-,\Omega^+(\un{d})\cup\Omega^0(\un{d})),
\end{equation}
where we used that (see (\ref{equ-inv}) and (\ref{sym-inv}))
\[\epsilon(\un{d},W^-\cup \Omega^+(\un{d})\cup \Omega^0(\un{d}))\leq \epsilon(\un{d})= \eta(\un{d})=
\eta(\un{d}, \Omega^-(\un{d}))= \epsilon(\un{d}, \Omega^-(\un{d})^c)=\]
\[=\epsilon(\un{d}, \Omega^+(\un{d})\cup\Omega^0(\un{d})).\]
Moreover, if the equality holds in (\ref{ine3}), then by (\ref{sym-inv})
\[\eta(\un{d})=\epsilon(\un{d},W^-\cup \Omega^+(\un{d})\cup \Omega^0(\un{d}))=
\eta(\un{d}, \Omega^-(\un{d})\setminus W^-),\]
which implies that $\Omega^-(\un{d})\setminus W^-\in S^-_{\un{d}}$ and hence that $W^-=\emptyset$
because of the minimality property of $\Omega^-(\un{d})$. Using the formula
\[\epsilon(\un{e},W)=\epsilon(\un{d},W)+ \Delta_0(\un{\chi(\Omega^+(\un{d}))})_W\]
and (\ref{for-cha}), the above inequalities (\ref{ine1}), (\ref{ine2}), (\ref{ine3})
give:
\begin{equation}\label{ine4}
\begin{sis}
& \epsilon(\un{e},W^+)=\epsilon(\un{d},W^+)-\val_{\Gamma\setminus S}(W^+,\Omega^+(\un{d})^c)
\leq  \epsilon(\un{d})-\val_{\Gamma\setminus S}(W^+,\Omega^+(\un{d})^c),\\
& \epsilon(\un{e},W^0)=\epsilon(\un{d},W^0)+ \val_{\Gamma\setminus S}(W^0,\Omega^+(\un{d}))\leq 0 ,\\
& \epsilon(\un{e},W^-)=\epsilon(\un{d},W^-)+ \val_{\Gamma\setminus S}(W^-,\Omega^+(\un{d}))
\leq -\val_{\Gamma \setminus S}(W^-,\Omega^0(\un{d})),\\
\end{sis}
\end{equation}
Using twice the additive formula (\ref{add-for}) for the disjoint union
$W=W^+\coprod W^0\coprod W^-$ and the above inequalities (\ref{ine4}),  we compute
\[\epsilon(\un{e},W)=\epsilon(\un{e},W^+)+\epsilon(\un{e},W^0)+\epsilon(\un{e},W^-)
+\val_{\Gamma\setminus S}(W^+, W^0) +\val_{\Gamma\setminus S}(W^+, W^-)+\]
\[+\val_{\Gamma\setminus S}(W^0, W^-)\leq \epsilon(\un{d})-
\val_{\Gamma\setminus S}(W^+, \Omega^0(\un{d})\setminus W^0) -
\val_{\Gamma\setminus S}(W^+, \Omega^-(\un{d})\setminus W^-)+\]
\begin{equation}\label{ine5}
-\val_{\Gamma\setminus S}(W^-, \Omega^0(\un{d})\setminus W^0)\leq
\epsilon(\un{d}).
\end{equation}
In particular, we have that $\epsilon(\un{e})\leq \epsilon(\un{d})$. If
the inequality in (\ref{ine5}) is attained for some $W\subseteq V(\Gamma)$,
i.e. if $\epsilon(\un{e})=\epsilon(\un{d})$,
 then also the inequalities
in (\ref{ine1}) and (\ref{ine3}) are attained for $W$, and we observed before that this
implies that
\begin{equation}\label{sharp1}
\begin{sis}
& W^+=\Omega^+(\un{d}), \\
& W^-=\emptyset.
\end{sis}
\end{equation}
Moreover, all the inequalities in (\ref{ine5}) are attained for $W$ and, substituting
(\ref{sharp1}), this implies that
\begin{equation}\label{sharp2}
\begin{sis}
& \val_{\Gamma\setminus S}(\Omega^+(\un{d}), \Omega^0(\un{d})\setminus W^0)=0, \\
&\val_{\Gamma\setminus S}(\Omega^+(\un{d}), \Omega^-(\un{d}))=0.
\end{sis}
\end{equation}
Since $\Gamma\setminus S$ is connected by hypothesis and $\Omega^+(\un{d})$ is a proper subset of
$V(\Gamma\setminus S)=V(\Gamma)$ because we fixed $\un{d}\not \in \Bss$ (see
(\ref{proper-Omega})), we deduce that  (using (\ref{sharp2})):
\[0<\val_{\Gamma\setminus S}(\Omega^+(\un{d}))=\val_{\Gamma\setminus S}(\Omega^+(\un{d}),
\Omega^-(\un{d})\cup \Omega^0(\un{d}))=\val_{\Gamma\setminus S}(\Omega^+(\un{d}), W^0).
\]
This gives that $W^0\neq \emptyset$, which implies that $W=W^+\cup W^0\supsetneq W^+=\Omega^+(\un{d})$
by (\ref{sharp1}). Since this holds for all $W\subseteq V(\Gamma)$ such that $\epsilon(\un e,W)=\epsilon (\un d) (=\epsilon (\un e))$, it holds in particular for $\Omega^+(\un e)$. Therefore, we get that $\Omega^+(\un{e})\supsetneq \Omega^+(\un{d})$ and the claim is proved.

\underline{STEP III: $\im(\ov{\pi})=\im(\pi)$.}

Let $\un{d}\in \Bss$, which by (\ref{proper-Omega}) is equivalent to have that $\epsilon(\un{d})=\eta(\un{d})=0$.
Let
\begin{equation*}
S_{\un{d},v_0}^-:=\{W\subseteq V(\Gamma)\: :\: \eta(\un{d},W)= \eta(\un d)=0\text{ and } v_0\in W\}.
\end{equation*}
The same proof as in Step II gives that $S_{\un{d},v_0}^-$ is
stable for the intersection (see (\ref{stab-inter})). Therefore, the set $S_{\un{d},v_0}^-$ admits a
minimum element
\begin{equation*}
\Omega^-(\un{d},v_0):=\bigcap_{W\in S_{\un{d},v_0}^-} W\subseteq V(\Gamma).
\end{equation*}
Note that, by the definition \ref{def-ss-qs}(\ref{qs-cond}), it follows that
\begin{equation}\label{proper-Omega-}
\un{d}\in \Bqs \Leftrightarrow \un{d}\in \Bss \text{ and }\Omega^-(\un{d},v_0)=V(\Gamma).
\end{equation}
Fix now an element $\un{d}\in \Bss\setminus \Bqs$ and consider the element
\begin{equation*}
\un{e}:=\un{d}-\Delta_0(\un{\chi(\Omega^-(\un{d},v_0))}).
\end{equation*}

\un{Claim:} The $0$-cochain $\un{e}$ satisfies the following two properties:
\begin{enumerate}[(i)]
\item $\eta(\un{e})=0$;
\item $\Omega^-(\un{e},v_0)\supsetneq \Omega^-(\un{d},v_0)$.
\end{enumerate}
The Claim concludes the proof of Step III. Indeed, property (i) says
that $\un{e}\in \Bss$  by (\ref{proper-Omega}) and therefore, by iterating
the above construction, we will find an element $\un{e'}\in \Bss$ such that
$\un{d}-\un{e'}\in \im(\Delta_0)$ and $\Omega^-(\un{e},v_0)=V(\Gamma)$,
which implies that $\ov{\pi}(\un{d})=\pi(\un{e'})$ and
$\un{e'}\in \Bqs$ by (\ref{proper-Omega-}). This shows that
$\im(\ov{\pi})=\im(\pi)$, q.e.d.

Let us now prove the Claim. Given any subset $W\subseteq V(\Gamma)$, we decompose
it as a disjoint union
\[W=W^-\coprod W^+,\]
where $W^-:=W\cap \Omega^-(\un{d},v_0)$ and $W^+:=W\setminus \Omega^-(\un{d},v_0)$.
Applying formula (\ref{add-for}) to the disjoint pair $(W^+, \Omega^-(\un{d},v_0))$
and using that $\eta(\un{d})=0$ and $\Omega^-(\un{d},v_0)\in S_{\un{d},v_0}^-$, we get
\[\eta(\un{d},W^+)=\eta(\un{d},W^+\cup \Omega^-(\un{d},v_0))-\eta(\un{d},\Omega^-(\un{d},v_0))
-\val_{\Gamma\setminus S}(W^+,\Omega^-(\un{d},v_0))\leq \]
\begin{equation}\label{ine6}
-\val_{\Gamma\setminus S}(W^+,\Omega^-(\un{d},v_0)).
\end{equation}
Applying again formula (\ref{add-for}) to the disjoint pair $(W^+,W^-)$ and using
$\eta(\un{d})=0$ and (\ref{ine6}), we get
\[\eta(\un{d},W)=\eta(\un{d},W^-)+\eta(\un{d}, W^+)+\val_{\Gamma\setminus S}(W^+,W^-)\leq
\]
\begin{equation}\label{ine7}
\leq 0 -\val_{\Gamma\setminus S}(W^+,\Omega^-(\un{d},v_0)) + \val_{\Gamma\setminus S}(W^+,W^-)=
-\val_{\Gamma\setminus S}(W^+,\Omega^-(\un{d},v_0)\setminus W^-).
\end{equation}
Using the formula
\begin{equation}\label{for-e}
\eta(\un{e},W)=\eta(\un{d},W)+ \Delta_0(\un{\chi(\Omega^-(\un{d},v_0))})_W
\end{equation}
and (\ref{for-cha}), the above inequality (\ref{ine7}) gives:
\[\eta(\un{e},W)=\eta(\un{d},W)-\val_{\Gamma\setminus S}
(W^-,(\Omega^-(\un{d},v_0)\cup W^+)^c)+\]
\begin{equation}\label{ine8}
+ \val_{\Gamma\setminus S}(W^+, \Omega^-(\un{d},v_0)\setminus W^-)\leq -
\val_{\Gamma\setminus S}(W^-,(\Omega^-(\un{d},v_0)\cup W^+)^c) \leq 0,
\end{equation}
which proves  part (i) of the Claim. Assume moreover that the inequality in (\ref{ine8})
is attained for some $W\subseteq V(\Gamma)$ such that $v_0\in W$. Then all
the inequalities must be attained also in (\ref{ine7}) and in particular $\eta(\un{d},W^-)=0$.
Since $v_0\in W\cap \Omega^-(\un{d},v_0)=W^-$, we deduce that  $W^-\in S_{\un{d},v_0}^-$ and
hence, by the minimality of $\Omega^-(\un{d},v_0)$, we get that $W^-=\Omega^-(\un{d},v_0)$.
It follows that $\Omega^-(\un{e},v_0)\supseteq \Omega^-(\un{d}, v_0)$. Using again
formulas (\ref{for-e}) and (\ref{for-cha}), together with the fact that
$\Omega^-(\un{d},v_0)\in S_{\un{d},v_0}^-$, we compute
\[\eta(\un{e},\Omega^-(\un{d},v_0))=\eta(\un{d}, \Omega^-(\un{d},v_0))-
\val_{\Gamma\setminus S}( \Omega^-(\un{d},v_0), \Omega^-(\un{d},v_0)^c)=\]
\[=-\val_{\Gamma\setminus S}( \Omega^-(\un{d},v_0), \Omega^-(\un{d},v_0)^c)<0,
\]
because $\Gamma\setminus S$ is connected by hypothesis and $\Omega^-(\un{d},v_0)$
is a proper subset of $V(\Gamma\setminus S)=V(\Gamma)$ by our initial assumption
$\un{d}\in \Bss\setminus \Bqs$ (see (\ref{proper-Omega-})) together with the
fact that $v_0\in \Omega^-(\un{d},v_0)$. Therefore $\Omega^-(\un{d},v_0)\not\in
S_{\un{e},v_0}^-$ and hence $\Omega^-(\un{e},v_0)\supsetneq \Omega^-(\un{d},v_0)$,
i.e. we get part (ii) of the Claim.
\end{proof}

\begin{remark}
The previous result was obtained for $S=\emptyset$ in \cite[Lemma 3.1.5]{Bus},
building upon ideas from \cite[Prop. 4.1]{caporaso}. Marco Pacini (\cite{Pac}) has communicated to us a different proof of the above result.
\end{remark}

By putting together Lemma \ref{empty}, Proposition \ref{card-qs} and equation
(\ref{lat-class}), we deduce the following

\begin{cor}\label{cor-card}
The cardinality of set $\Bqs$ is equal to the complexity $c(\Gamma\setminus S)$ of $\Gamma\setminus S$.
In particular, $\Bqs\neq \emptyset$ if and only if $\Gamma\setminus S$ is connected.
\end{cor}

\end{nota}

\section{Fine compactified Jacobians and N\'eron models}

Let $f:\X\to B=\Spec(R)$ be a one-parameter   regular local smoothing of $X=\X_k$ (see \ref{notner}).
Fix a section section $\sigma:B\to \X$ and a polarization $\un q$ on $X$ (see \ref{pola-nota}) such that
$d:=|\un q|$.
Consider the $B$-scheme $J_f^{\sigma}(\un q)$ of \ref{finecoarse} and denote by   $\Jsm$ its smooth locus over $B$.


\begin{thm}
\label{picner}
Let $f: \X\la B$ be a one-parameter  regular local  smoothing of $X=\X_k$. Let $\sigma$ be a section of $f$
and $\un q$ a polarization on $X$ such that $d:=|\un q|$.
Then $J_f^{\sigma}(\un q)_{\rm sm}$ is isomorphic to the N\'eron model $\nerd$ of the degree-$d$ Jacobian of the generic fiber $\X_K$ of $f$.
\end{thm}
\begin{proof}
According to Fact \ref{smooth}(\ref{smooth1}), $\Jsm$ parametrizes line bundles on $\X$ of relative degree $d$ and whose
special fiber is $\un q$-P-quasistable, where $P:=\sigma(\Spec k)\in \Xsm$. If we denote by   $v_0$ the vertex of the dual graph $\Gamma_X$
of $X$ corresponding to the irreducible component to which $P$ belongs, then the $\un q$-P-quasistable multidegrees on $X$
correspond to the $0$-cochains belonging to $B_{\Gamma_X}^{v_0}(\un q)$ in the notation of Definition \ref{def-ss-qs}. Therefore, we get a canonical $B$-isomorphism
\begin{equation}
\label{pfclaim}
\Jsm \cong \frac{\coprod_{\md \in B_{\Gamma_X}^{v_0}(\un q)}\picf{\md} }{\sim_K},
\end{equation}
where $\sim_K$ denotes the gluing along the general fibers of $\picf{\un d}$ which are isomorphic to $\Pic^d(\X_K)$.
Since the general fiber of $\Jsm$ is isomorphic
to $\Pic^d(\X_K)$, the N\'eron mapping property gives a map (see Fact \ref{expl-Neron}):
\[r:\Jsm\cong \frac{\coprod_{\md \in B_{\Gamma_X}^{v_0}(\un q)}\picf{\md} }{\sim_K}\longrightarrow  \nerd\cong \frac{\coprod_{\delta \in \Delta_X^d}\picf{\delta} }{\sim_K}.
\]
Since we have a natural inclusion $i:\Jsm\hookrightarrow \picf{d}$ which is the identity on the general fibers, the map $r$ factors
through the map $q$ of (\ref{sepquot}).
Therefore the map $r$ sends each $\picf{\un d}$ into $\picf{[\un d]}$. Since the natural map
$ B_{\Gamma_X}^{v_0}(\un q)\to \Delta_X^d$ is a bijection according to Proposition \ref{card-qs}, we conclude that the map $r$ is an isomorphism.


\end{proof}



\begin{remark}
\noindent
\begin{enumerate}[(i)]
\item In the terminology of \cite[Def. 2.3.5]{capsurvey} and  \cite[Def. 1.4 and Prop. 1.6]{capNtype}, the above Theorem \ref{picner} says that the fine compactified
Jacobians $J_X^P(\un q)$ are always of N\'eron-type (or N-type).
\item  Using Theorem \ref{nondeg-thm} and Remark \ref{can-nondeg},  the above Theorem \ref{picner} recovers \cite[Thm 2.9]{capNtype}, which is
a generalization of  \cite[Thm. 6.1]{capneron}:  $\ov{P_X^d}$ is of N\'eron-type  if $X$ is weakly $d$-general.
\end{enumerate}
\end{remark}

\section{A stratification of the fine compactified Jacobians}\label{Strat-sec}


In the present section we shall exhibit a stratification of  $J^P_X(\un q)$ in terms of fine compactified
Jacobians of partial normalizations of $X$.


For each subset $S\subseteq \Xsing$, denote by $\JPS$ the subset of $\JPX$ corresponding to torsion-free sheaves which are not free exactly at $S$.
Each $\JPS$ is a locally closed subset of $J_X^P(\un q)$ that we endow with the reduced schematic structure.
Similarly, we endow the closure $\ov{\JPS}$ of each stratum $\JPS$ with the reduced schematic structure.
We have the following stratification
\begin{equation}\label{stratjac}
\JPX=\coprod_{S\subseteq \Xsing }\JPS.
\end{equation}

\begin{thm}\label{strata-Jac}
The stratification of $\JPX$ given in (\ref{stratjac}) satisfies the following properties:
\begin{enumerate}[(i)]
\item \label{strata-Jac1} Each stratum $\JPS$ is a disjoint union of $c(\Gamma_{X_S})$ torsors for the generalized Jacobian $J(X_S)$ of the partial normalization of
$X$ at $S$. In particular, $\JPS$ is non-empty if and only if $X_S$ is connected.
\item \label{strata-Jac2} The closure of each stratum is given by
\[\ov{\JPS}=\coprod_{S\subset S'} J^P_{X, S'}(\un q).\]
\item \label{strata-Jac3} The pushforward $(\nu_S)_*$ along the partial normalization map $\nu_S:X_S\to X$ gives isomorphisms:
\[\begin{sis}
& J_{X_S}^P(\un {q^S})_{\rm sm}\cong J_{X,S}^P(\un q), \\
& J_{X_S}^P(\un {q^S})\cong \ov{J_{X,S}^P(\un q)},\\
\end{sis}
\]
where $\un{q^S}$ is the polarization on $X_S$ defined in Lemma-Definition \ref{pola-norma} and $P$ is seen as a smooth point of $X_S$
using the isomorphism  $(X_S)_{\rm sm}\cong \Xsm$.
\end{enumerate}
\end{thm}

\begin{remark}
It is easy to see that if $\un q$ is the canonical polarization of degree $d$ (see Remark \ref{inequ}\eqref{inequ2}) then $\un{q^S}$ is again a canonical polarization
for every $S\subseteq \Xsing$ if and only if $d=g-1$.
This explains why the stratification found by Caporaso for $\ov{P_X^{g-1}}$ in \cite[Sec. 4.1]{captheta} can work only in degree $d=g-1$. In the general case, even if
one is interested only in coarse or fine compactified Jacobians with respect to canonical polarizations, non-canonical polarizations naturally show-up in
the above stratification.
\end{remark}


Before proving the theorem,  we need to analyze the multidegrees of the sheaves $\I$ belonging to the strata $\JPS$.

\begin{nota}{\emph{Multidegrees of sheaves $\I\in J_X^P(\un{q})$}}

For a torsion-free, rank $1$ sheaf $\I$ on $X$, the subset $NF(\I)\subset \Xsing$ where $\I$ is not free (see \ref{notsheaves})
 admits a partition
\[NF(\I)=NF_e(\I)\coprod NF_i(\I),\]
where $\NF_e(\I):=\NF(\I)\cap X_{\rm ext}$ and $\NF_i(\I):=\NF(\I)\cap X_{\rm int}$.

Given a sheaf $\I$ on $X$, we define its multidegree $\un{\deg(\I)}$ as the $0$-cochain
in $C^0(\Gamma_X,\Z)$ such that $\un{\deg(\I)}_{v}:=\deg_{X[v]}(\I)$ for every $v\in V(\Gamma_X)$.
Given a subset $W\subset V(\Gamma_X)$, we define
\[\un{\deg(\I)}_W:=
\sum_{v\in V(\Gamma_{X[W]})} \un{\deg(\I)}_{v}= \sum_{v\in V(\Gamma_{X[W]})} \deg_{X[v]}(\I).\]
In what follows we analyze the difference between $\deg_{X[W]}(\I)$ and $\un{\deg(\I)}_W$
where $\I$ is a torsion-free, rank $1$ sheaf on $X$.

\begin{lemma}\label{add-deg}
Let $Y$ be a subcurve of $X$ and let $Y_1,\cdots, Y_m$ be the irreducible components of
$Y$. Then
\[\deg_{Y}(\I)=\sum_{i=1}^m \deg_{Y_i}(\I)+ |NF_e(\I)\cap X\setminus Y^c|. \]
\end{lemma}
\begin{proof}
We will first prove that if $Y$ and $Z$ are two subcurves
of $X$ without common irreducible components then
\begin{equation}\label{add-deg-2}
\deg_{Y\cup Z}(\I)=\deg_{Y}(\I)+\deg_{Z}(\I)+|NF(\I)\cap Y\cap Z|.
\end{equation}
Using Proposition \ref{sheaf-linebun}(\ref{sheaf-linebun1}), there exists a line bundle $L$ on $X_S$ where $S=\NF(\I)$ such that $\I=(\nu_S)_*(L)$.
By Proposition \ref{sheaf-linebun}(\ref{sheaf-linebun3}), we have the equalities
\begin{equation*}
\begin{sis}
& \deg_{Y\cup Z}\I=\deg_{Y_S\cup Z_S} L+|S_i^{Y\cup Z}|, \\
& \deg_{Y}\I=\deg_{Y_S} L+|S_i^{Y}|, \\
& \deg_{Z}\I=\deg_{Z_S} L+|S_i^{Z}|. \\
\end{sis}\tag{a}
\end{equation*}
Since $L$ is a line bundle, we have that
\begin{equation*}
\deg_{Y_S\cup Z_S}L=\deg L_{|Y_S\cup Z_S}=\deg L_{|Y_S}+\deg L_{|Z_S}=\deg_{Y_S}L+\deg_{Z_S}L. \tag{b}
\end{equation*}
We have already observed in (\ref{equa-S}) that
\begin{equation*}
|S_i^{Y\cup Z}|=|S_i^Y|+|S_i^Z|+|S\cap Y\cap Z|. \tag{c}
\end{equation*}
The equation (\ref{add-deg-2}) is easily proved by putting together equations (a), (b) and (c).

The proof of the lemma is now by induction on the number $m$ of irreducible components of $Y$. If $m=1$ then the formula follows from the fact
that $X\setminus Y_1^c$ contains only internal nodes. As for the induction step,
using (\ref{add-deg-2}), we can write
\begin{equation*}
\deg_Y(\I)=\deg_{Y_1\cup\cdots \cup Y_{m-1}}(\I)+\deg_{Y_m}(\I)+|NF_e(\I)\cap
(Y_1\cup\cdots\cup Y_{m-1})\cap Y_m|.
\tag{*}
\end{equation*}
By the induction hypothesis, we have that
\begin{equation*}
\deg_{Y_1\cup\cdots \cup Y_{m-1}}(\I)=\sum_{i=1}^{m-1}\deg_{Y_i}(\I)+|NF_e(\I)\cap
X\setminus (Y_1\cup\cdots\cup Y_{m-1})^c|.
\tag{**}
\end{equation*}
Since an external node in $X\setminus Y^c$ either is an external node of
$Y_1\cup\cdots\cup Y_{m-1}$ or is node at which $Y_m$ intersects
$Y_1\cup \cdots \cup Y_{m-1}$, we have that
\begin{equation*}
|NF_e(\I)\cap X\setminus Y^c|=|NF_e(\I)\cap X\setminus (Y_1\cup\cdots\cup Y_{m-1})^c|+
\tag{***}
\end{equation*}
\[+ |NF_e(\I)\cap (Y_1\cup\cdots\cup Y_{m-1})\cap Y_m)|.\]
We conclude by putting together (*), (**), (***).
\end{proof}

For every subset $S\subseteq \Xsing$, denote by   $B^P_{X,S}(\un q)$ the set of possible multidegrees of sheaves $\I\in \JPS$.
Write $S=S_e\coprod S_i$, where $S_e:=S\cap X_{\rm ext}$ and $S_i=S\cap X_{\rm{int}}$.
We need the following version of the dual graph of $X$: the \emph{loop-less} dual graph of $X$, denoted by $\w{\Gamma_X}$,
is the graph obtained from $\Gamma_X$ by removing all
the loops. In particular,
$V(\w{\Gamma_X})=V(\Gamma_X)$ while $E(\w{\Gamma_X})$ can be identified with $X_{\rm int}$.



\begin{prop}\label{multdeg-sh}
For any $S\subseteq \Xsing$ we have that
\[B^P_{X,S}(\un q)=B_{\w{\Gamma_X}\setminus S_e}^{v_P}(\un{q}).\]
In particular, the cardinality of $B^P_{X,S}(\un q)$ is equal to $c(\w{\Gamma_X}\setminus S_e)=c(\Gamma_X\setminus S)=
c(\Gamma_{X_S})$.
\end{prop}
\begin{proof}
Consider the loop-less dual graph $\w{\Gamma_X}$ of $X$ and a sheaf $\I\in J_X^P(\un{q})$. Then, Lemma \ref{add-deg} translated in terms
of $\w{\Gamma_X}$ says that, for every $W\subset V(\Gamma_X)=V(\w{\Gamma_X})$, the multidegree $\un{\deg(\I)}$
of $\I$ satisfies:
\[\deg_{X[W]}(\I)=\un{\deg(\I)}_W+|NF_e(\I)\cap \w{\Gamma_X}[W]|.\]
In particular, $\deg(\I)=|\un{\deg(\I)}|+|NF_e(\I)|$. Using this formula together with the fact
that, for every $W\subset V(\Gamma_X)=V(\w{\Gamma_X})$, $\delta_{X[W]}=\val_{\w{\Gamma_X}}(W)$ we deduce that a torsion-free, rank $1$ sheaf $\I$
is $P$-quasistable with respect to $\un{q}$ (in the sense of Definition \ref{sheaf-ss-qs}(\ref{sheaf-qs}))
if and only if  its multidegree $\un{\deg(\I)}\in C^0(\w{\Gamma_X},\Z)$ is $v_P$-quasistable with respect to $\un{q}$
(in the sense of Definition \ref{def-ss-qs}(\ref{qs-cond})). The last assertion follows from
Corollary \ref{cor-card} together with the easy facts that the operation of
removing loops from a graph does not change its complexity and that $\Gamma\setminus S=\Gamma_{X_S}$.

\end{proof}


\end{nota}




\begin{proof}[Proof of Theorem \ref{strata-Jac}]
Part (i): By Proposition \ref{sheaf-linebun}(\ref{sheaf-linebun1}), the subvariety of $\JPS$ consisting of sheaves with a fixed multidegree $\un d$ is
isomorphic to $\Pic^{\un d'}(X_S)$, where $\un d'$ is related to $\un d$ according to the formula of  Proposition \ref{sheaf-linebun}(\ref{sheaf-linebun3}).
Each  $\Pic^{\un d'}(X_S)$ is clearly a torsor for $J(X_S)$. We conclude by the fact that  the set $B_{X,S}^P(\un q)$ of multidegrees of sheaves belonging to $\JPS$ has cardinality
$c(\Gamma_{X_S})$ by Proposition \ref{multdeg-sh}.

Part (ii): The inclusion
\[\ov{\JPS}\subset \coprod_{S\subset S'} J^P_{X, S'}(\un q)\]
is clear since under specialization the set $\NF(\I)$ can only increase.
 In order to prove the reverse inclusion, it is enough to show that if $\I\in J_X^P(\un q)$ is such that
$n\in \NF(\I)$ then there exists a sheaf $\I'\in J_X^P(\un q)$ specializing to $\I$ and such that $\NF(\I')=\NF(\I)\setminus \{n\}$.

Suppose first that $n$ is an external node and, up to reordering the components of $X$, assume that $n\in C_1\cap C_2$.
By looking at the miniversal deformation ring of $\I$ (see e.g. \cite[Lemma 3.14]{CMKV2}),
we can find a torsion free, rank $1$ sheaf $\I'$ specializing to
$\I$ with $\NF(\I')=\NF(\I)\setminus \{n\}$ and such that the multidegree of $\I'$ is related to the one of $\I$ by means of the following
\begin{equation}\label{smoothing-node1}
\deg_{C_i} \I'=\begin{cases}
\deg_{C_1} \I +1 & \text{Êif } i=1,\\
\deg_{C_i}\I  & \text{ if } i\neq 1.
\end{cases}
\end{equation}
Since the condition of being $\un q$-P-quasistable is an open condition, we get that $\I'$ is $\un q$-P-quasistable
and we are done.


Suppose now that $n$ is an internal node.
By looking at the miniversal deformation ring of $\I$, we can find a torsion-free rank $1$ sheaf $\I'$ specializing to $\I$
with $\NF(\I')=\NF(\I)\setminus \{n\}$ and such that the multidegree of $\I'$ is equal to the one of $\I$.
Clearly $\I'$ is $\un q$-$P$-quasistable and we are done.

Part (iii): First of all, observe that the pushforward map $(\nu_S)_*$ is a closed embedding since 
it is induced by a functor between the categories of torsion-free rank one sheaves on $X_S$ and on $X$ which
is fully faithful, as it follows from \cite[Lemma 3.4]{EGK} (note that the result in loc. cit. extends easily from the case
of integral curves to the case of reduced curves).
\footnote{We are grateful to Eduardo Esteves for pointing out to us this argument.}
Therefore, in order to conclude the proof of part (iii), it is enough to show that the map $(\nu_S)_*$ 
induces a bijection on geometric points.

Consider first the bijection of Proposition \ref{sheaf-linebun}(\ref{sheaf-linebun1}). We claim that a line bundle $L\in \Pic(X_S)$ is
$\un{q^S}$-$P$-quasistable on $X_S$ if and only if $(\nu_S)_*L$ is $\un{q}$-$P$-quasistable on $X$. This amounts to prove
that for any  subcurve $Y\subset X$ we have
\[\deg_{Y_S}L\geq  \un{q^S}_{Y_S}-\frac{\delta_{Y_S}}{2}\Longleftrightarrow \deg_Y (\nu_S)_*L\geq \un{q}_{Y}-\frac{\delta_Y}{2},\]
and similarly with the strict inequality $>$ (since $P\in Y$ if and only if $P\in Y_S$).
This equivalence follows from the equalities
\begin{equation*}
\begin{sis}
&\deg_{Y_S} L=\deg_Y (\nu_S)_*L-| S_i^Y|, \\
& \un {q^S}_{Y_S}=\un q_Y-\frac{|S_e^Y|}{2}- |S_i^Y|, \\
& \delta_{Y_S}=\delta_Y-|S_e^Y|,\\
\end{sis}
\end{equation*}
where the first equality follows from Proposition \ref{sheaf-linebun}(\ref{sheaf-linebun3}), the second follows from the definition of $\un{q^S}$
(see Lemma-Definition \ref{pola-norma}) and the third is easily checked.
Therefore, using Fact \ref{smooth}(\ref{smooth1}), the push-forward via the normalization map $\nu_S$ induces a morphism
\begin{equation}\label{map-smooth}
(\nu_S)_*:J_{X_S}^P(\un{q^S})_{\rm sm} \to J_{X,S}^P(\un q),
\end{equation}
which is bijective on geometric points. 
This proves the first isomorphism in Part \eqref{strata-Jac3}.

Let us now prove the second isomorphism of Part \eqref{strata-Jac3}. To that aim, consider  two subsets $\emptyset\subseteq S\subseteq S'\subseteq \Xsing$. We have a commutative diagram
\[\xymatrix{
X_{S'} \ar[rr]^{\nu_{S'\setminus S}}\ar[dr]_{\nu_{S'}}& & X_S\ar[dl]^{\nu_S}\\
& X &\\
}
\]
where $\nu_{S'\setminus S}$ is the partial normalization of $X_S$ at the nodes corresponding to $S'\setminus S$. By abuse of notation, we denote by     $P$
the inverse image of $P\in X$ in $X_S$ and in $X_{S'}$.  We claim that the above diagram induces,
via push-forwards, a commutative diagram
\begin{equation}\label{push-2norma}
\xymatrix{
J_{X_{S'}}^P(\un{q^{S'}})_{\rm sm} \ar[rd]_{(\nu_{S'})_*}^{\cong}\ar[rr]^{(\nu_{S'\setminus S})_*}_{\cong} & & J_{X_S, S'\setminus S}^P(\un{q^S}) \ar[dl]^{(\nu_S)_*}_{\cong} \\
& J_{X,S'}^P(\un q) &
}
\end{equation}
where all the maps are isomorphisms. Indeed, from \eqref{map-smooth} with $S$ replaced by $S'$, it follows that the map $(\nu_{S'})_*$ is an isomorphism.
Similarly, if we apply \eqref{map-smooth} with $X$ replaced by $X_S$, $S$ replaced by $S'\setminus S$ and $\un q$ replaced by $\qS$, we obtain that $(\nu_{S'\setminus S})_*$
is an isomorphism since it is easily checked that $(X_S)_{S'\setminus S}\cong X{_{S'}}$ and $\un{(\qS)^{S'\setminus S}}=\un{q^{S'}}$. Since the diagram \eqref{push-2norma} is clearly
commutative, we get that $(\nu_S)_*$ is well-defined and that it is an isomorphism.

From the fact that the map $(\nu_S)_*$ in diagram \eqref{push-2norma} is an isomorphism, using the stratification \eqref{stratjac} and the one in part \eqref{strata-Jac2}, we deduce that the natural map
\begin{equation}\label{map-sing}
(\nu_S)_*: J_{X_S}^P(\qS)=\coprod_{S\subseteq S'\subseteq \Xsing} J^P_{X_S,S'\setminus S}(\qS)
\to \coprod_{S\subseteq S'\subseteq \Xsing} J_{X,S'}^P(\un q)= \ov{J_{X,S}^P(\un q)}
\end{equation}
is bijective on geometric points, which concludes the proof.

\end{proof}

\begin{cor}\label{cor-strat}
For the stratification in (\ref{stratjac}), it holds:
\begin{enumerate}[(i)]
\item $\JPS$ has pure codimension equal to $|S|$.
\item $\ov{\JPS}\supset  J^P_{X, S'}(\un q)$
if and only if $S\subseteq S'$.
\item The smooth locus of $\ov{\JPS}$ is equal to $\JPS$.
\end{enumerate}
\end{cor}
\begin{proof}
Part (i) follows from Theorem \ref{strata-Jac}(\ref{strata-Jac1}) together with the equality
\[\dim J(X)-\dim J(X_S)=g(X)-g(X_S)=|S|,\]
where we used that $X_S$ is connected.

Part (ii) follows from Theorem \ref{strata-Jac}(\ref{strata-Jac2}).

Part (iii) follows from Theorem \ref{strata-Jac}(\ref{strata-Jac3}).
\end{proof}

\begin{remark}
A result similar to Corollary \ref{cor-strat} was proved by Caporaso in \cite[Thm. 6.7]{capneron} for the compactified Jacobian
$\ov{P_X^d}$ (see Remark \ref{nota-Jac}(\ref{nota-Jac3})) of a $d$-general curve $X$ in the sense of Remark \ref{can-gen}.
Indeed, by using Theorem \ref{nondeg-thm}, our Corollary \ref{cor-strat} recovers \cite[Thm. 6.7]{capneron} and extends
it to the case of $X$ weakly $d$-general in the sense of Remark \ref{can-nondeg}.
\end{remark}

\section{Fine compactified Jacobians as quotients}
\label{quotient}

\begin{nota}
\label{ladder}
Recall from \ref{notnodal} that we denote by     $\XSh$ (resp. $\hX$) the partial blowup of $X$ at
$S\subseteq \Xsing$ (resp. the total blowup of $X$) and the natural blow-down morphisms by
$\pi_S:\XSh\to X$ (resp. $\pi:\hX\to X$). Moreover, for each $S\subseteq \Xsing$,
we have a commutative diagram
\begin{equation}\label{part-fac}
\xymatrix{
\hX \ar[rr]^{\pi^S}\ar[dr]_{\pi} & & \XSh \ar[dl]^{\pi_S} \\
& X &
}
\end{equation}
where $\pi^S$ is the blow-down of all the exceptional subcurves of $\hX$ lying over the nodes
of $\Xsing\setminus S$.

Given a polarization $\un q$ on $X$, consider the polarizations $\widehat{\un{q^S}}$ (resp.
$\widehat{\un q}$) on $\XSh$ (resp. $\hX$) introduced in Lemma-Definition \ref{pola-blow}.
Given $P\in\Xsm$, we denote also with $P$ the inverse image of $P$ in $\XSh$ and in $\widehat X$,
in a slight abuse of notation.



Given $S\subseteq \Xsing$, denote by     $\Jprim$ the open and closed subset of $J_{\hXS}^P(\qSh)_{\rm sm}$
consisting of all line bundles that have degree $-1$ on all the exceptional components of $\hXS$.
Note that $\Jprim$ may be empty for some $S\subseteq \Xsing$.

\begin{thm}\label{quot-blow}
\noindent
\begin{enumerate}[(i)]
\item \label{quot-blow1}
For any $S\subseteq \Xsing$, $\Jprim$ is a disjoint union of $c(\Gamma_{X_S})$ torsors for the
generalized Jacobian $J(\hXS)\cong J(\hX)\cong J(X)$.
In particular $\Jprim$ is non-empty if and only if $X_S$ is connected.
\item \label{quot-blow2}
The pull-back via the map $\pi^S$ induces an open and closed embedding
\begin{equation}\label{emb-prim}
(\pi^S)^*: \Jprim \hookrightarrow J^P_{\hX}(\qh)_{\rm sm}.
\end{equation}
Via the above identification, $J^P_{\hX}(\qh)_{\rm sm}$ decomposes into a disjoint union of
open and closed strata
\begin{equation}\label{strat-blow}
J^P_{\hX}(\qh)_{\rm sm}=\coprod_{\stackrel{\emptyset\subseteq S\subseteq \Xsing}{}
} \Jprim.
\end{equation}
\item \label{quot-blow3} The push-forward along the map $\pi$ induces a surjective morphism
\[\pi_*:J^P_{\hX}(\qh)_{\rm sm} \twoheadrightarrow J^P_{X}(\un q),\]
which is compatible with the stratifications \eqref{stratjac} and \eqref{strat-blow} in the sense
that it induces a cartesian diagram
\[
\xymatrix{
\Jprim \ar@{^{(}->}[r]^>>>>{{(\pi^S)}^*} \ar@{->>}[d]_{(\pi_S)_*}& J^P_{\hX}(\qh)_{\rm sm} \ar@{->>}[d]^{\pi_*}\\
J_{X,S}^P(\un q) \ar@{^{(}->}[r] & J_X^P(\un q)
}
\]
Moreover, the map $(\pi_S)_*$ on the left hand-side of the above diagram is given by taking a quotient by the algebraic
torus $\Gm^{|S|}$ of dimension $|S|$.
\end{enumerate}
\end{thm}

\begin{proof}
Let us start by proving Part (ii).
First of all, observe that the pull-backs via the maps of diagram \eqref{part-fac} induce
canonical isomorphisms between the  generalized Jacobians
\[\pi^*:J(X)\stackrel{\cong}{\longrightarrow} J(\XSh) \stackrel{\cong}{\longrightarrow} J(\hX),\]
so that we will freely identify them during this proof.

Let us prove that the map \eqref{emb-prim} is well-defined, that is, given
a $P$-$\qSh$-quasistable line bundle $L$ on $\XSh$, then
$(\pi^S)^*L$ is a $P$-$\qh$-quasistable line bundle on $\hX$. Clearly we have that
$\deg (\pi^S)^*L~=\deg L=|\qSh|=|\qh|.$ Moreover, if $Z$ is a subcurve of $\hX$ and we denote by
$\pi^S(Z)$ its image in $\XSh$, then it is easily checked that $\delta_Z\geq \delta_{\pi^S(Z)}$, which
implies that
\[\deg_Z (\pi^S)^*L=\deg_{\pi^S(Z)}L\geq \qSh_{\pi^S(Z)}-\frac{\delta_{\pi^S(Z)}}{2}\geq
\qh_Z-\frac{\delta_Z}{2},\]
where the first inequality is strict if $P\in \pi^S(Z)$ which happens if and only if
$P\in Z$. Hence, $(\pi^S)^*L$ is a $P$-$\qh$-quasistable.

The map \eqref{emb-prim} is equivariant with respect to the action of the generalized
Jacobians $J(\XSh)\cong J(\hX)$ and both the sides are disjoint union of torsors for
these generalized Jacobians. Therefore, $\Jprim$ is mapped via
\eqref{emb-prim} isomorphically onto a disjoint union of connected components of
$J^P_{\hX}(\qh)_{\rm sm}$. The image of $\Jprim$ inside $J^P_{\hX}(\qh)_{\rm sm}$
consists of all $P$-$\qh$-quasistable line bundles on $\hX$ that have degree $-1$ on the
exceptional components lying over the nodes belonging to $S$ and degree $0$ on the other
exceptional components.

In order to prove that the decomposition description \eqref{strat-blow} holds, it remains to show
that any line bundle $L$ on $\hX$ which is $P$-$\qh$-quasistable must have degree $-1$ or $0$
on each exceptional component $E$ of $\hX$.
Indeed, by applying (\ref{multdeg-sh1}) to $E$ and to $E^c=\ov{\hat X\setminus E}$ and using that
$\delta_E=2$, we get that $\deg_EL$ must be equal to $-1$, $0$ or $1$.
However, since $P\in E^c$, strict inequality must hold when applying (\ref{multdeg-sh1}) to  $E^c$, so $\deg_E L$
can not be equal to $1$. Part \eqref{quot-blow2} is now complete.

\underline{CLAIM: } The commutative diagram (\ref{diag-S}) induces a commutative diagram
\begin{equation}\label{diag-norm-blow}
\xymatrix{
J_{X_S}^P(\qS)_{\rm sm} \ar[rd]^{\cong}_{(\nu_S)_*} & & \Jprim \ar@{->>}[dl]^{(\pi_S)_*} \ar@{->>}[ll]_{i_S^*} \\
& J_{X,S}^P(\un q) &
}
\end{equation}
where $(\nu_S)_*$ is an isomorphism  and the maps $i_S^*$ and
$(\pi_S)_*$ are surjective. The fact that the map $(\nu_S)_*$ is well-defined and is an isomorphism is proved in Theorem \ref{strata-Jac}\eqref{strata-Jac3}.
Therefore, the commutativity of the diagram, together with the fact that it  is well-defined, will follow from  Proposition \ref{sheaf-linebun}\eqref{sheaf-linebun1}
if we show that $i_S^*$ is well-defined, i.e. if $L$ is a $P$-$\qSh$-quasistable line bundle on $\XSh$ having degree $-1$ on each exceptional component of $\XSh$
then  $i_S^*(L)$ is a $P$-$\qS$-quasistable line bundle on $X_S$.
Indeed, we have that
\[\deg i_S^*(L)=\deg L- |S|=|\qSh|-|S|=|\un q|-|S|=|\qS|.\]
Moreover, for any subcurve $Y_S\subseteq X_S$, it is easily checked that  (in the notations of Lemma-Definition \ref{pola-norma})
\[
\begin{sis}
&\deg_{Y_S} i_S^*(L)=\deg_{i_S(Y_S)}L, \\
& \un{q^S}_{Y_S}= \un{q}_Y-\frac{|S_e^Y|}{2}-|S_i^Y|= \qSh_{i_S(Y_S)}-\frac{|S_e^Y|}{2}-|S_i^Y|, \\
& \delta_{Y_S}=\delta_Y-|S_e^Y|= \delta_{i_S(Y_S)}-2|S_i^Y|-|S_e^Y|. \\
\end{sis}
\]
Using the above relations, it turns out that the inequality \eqref{multdeg-sh1} for the subcurve $Y_S\subseteq X_S$ and the line bundle $i_S^*L$ follows
form the same inequality \eqref{multdeg-sh1} applied to the subcurve $i_S(Y_S)\subseteq \XSh$ and the line bundle $L$. Hence $i_S^*$ is well-defined.

In order to conclude the proof of the claim, it remains to prove that the map $i_S^*$ is surjective.
Clearly $J_{X_S}^P(\qS)_{\rm sm}$ is a disjoint union of torsors for $J(X_S)$ of the form $\Pic^{\un d'}(X_S)$
for some suitable multidegrees $\un d'$; the number of such components  is $c(\Gamma_{X_S})$ by Theorem \ref{strata-Jac}.
Similarly, $\Jprim$ is a disjoint union of torsors for $J(\XSh)$ of the form $\Pic^{\un d}(\XSh)$ for some suitable multidegrees $\un d$ on $\XSh$;
call $n_S$ the number of such components.
It is clear that the map $i_S^*$ is equivariant with respect to the actions of $J(X_S)$ and
$J(\XSh)$ and of the natural surjective map
\begin{equation}\label{quot-gen-Jac}
J(\XSh)\twoheadrightarrow J(X_S).
\end{equation}
This implies that each connected component $\Pic^{\un d}(\XSh)$ of $\Jprim$ is sent surjectively onto the connected component $\Pic^{\un d_{X_S}}(X_S)$
of $J_{X_S}^P(\qS)_{\rm sm}$, where $\un d_{X_S}$ is the restriction of the multidegree $\un d$ to $X_S$. Since $\un d$ has degree $-1$ on each exceptional
component of $\XSh$, the multidegree $\un d$ is completely determined by its restriction $\un d_{X_S}$. This means that different components of $\Jprim$ are sent
to different components of $J_{X_S}^P(\qS)$. In particular, we get that
\begin{equation*}
n_S\leq c(\Gamma_{X_S}). \tag{*}
\end{equation*}

Let us now show that $n_S=c(\Gamma_{X_S})$, which will conclude the proof of the  Claim and also the proof of Part \eqref{quot-blow1}.
By Theorem \ref{picner} and Fact \ref{expl-Neron}, it follows that the number of connected components of $J_{\hX}^P(\qh)_{\rm sm}$ is equal to $c(\Gamma_{\hX})$.
Using the decomposition \eqref{strat-blow} and the inequality (*), we get that
\begin{equation*}
c(\Gamma_{\hX})=\sum_{\emptyset \subseteq S\subseteq \Xsing} n_S\leq \sum_{\emptyset \subseteq S\subseteq \Xsing}  c(\Gamma_{X_S}). \tag{**}
\end{equation*}
Fact \ref{compl-blow} applied to the graph $\Gamma=\Gamma_{\hX}$ and $S=E(\Gamma_X)$  gives that equality must hold in (**) and hence, a fortiori, also in (*)
for every $S\subset \Xsing$. Part (i) follows.

Finally, let us prove Part \eqref{quot-blow3}. The image of the stratum $\Jprim\subset J_{\hX}^P(\qh)_{\rm sm}$ via $\pi_*$ coincides with its image via the map $(\pi_S)_*$, which
by the above Claim, is equal to $J_{X,S}^P(\un q)$. Therefore $\pi_*$ is surjective and compatible with the filtrations \eqref{stratjac} and \eqref{strat-blow}.
For all the subsets $S\subseteq \Xsing$ such that $\Jprim\neq \emptyset$, the map $(\pi_S)_*$ is given by taking the quotient by the kernel of the surjection \eqref{quot-gen-Jac},
which is equal to $\Gm^{|S|}$ since $X_S$ is connected by Part \eqref{quot-blow1}. The proof is now complete.




\end{proof}

\end{nota}

\begin{nota}{\emph{Relating one-parameter regular local smoothings of $X$ and of $\hX$}}
\label{two}

Let $f:\X \la \Spec R=B$ be a one-parameter regular local smoothing of $X$ (see \ref{notner})
and assume that $f$ admits a section $\sigma$.

Then, as shown in \cite[Sec. 8.4]{capneron}, there exists a one-parameter regular local smoothing $\wh{f}:\wh{\X}\to B_1$ of $\hX$ endowed with a section
$\wh{\sigma}$ in such a way that there is a commutative diagram
\begin{equation}\label{compa-smooth}
\xymatrix{
{\wh{\X}}  \ar[d]^{\wh{f}} \ar[r] &\X \ar[d]_{f} \\
B_1 \ar[r]  \ar^{\wh{\sigma}}@/^/[u] & {B} \ar_{\sigma}@/_/[u]
}
\end{equation}
which, moreover, is a cartesian diagram on the general fibers of $f$ and $\wh f$.

For the reader's convenience, we review Caporaso's construction.
Let $t$ be a uniformizing parameter of $R$ (i.e. a generator of the maximal ideal of $R$) and consider the degree-$2$ extension $K\hookrightarrow K_1:=K(u)$ where $u^2=t$.
Denote by $R_1$ the integral closure of $R$ inside $K_1$ so that  $B_1:=\Spec(R_1)\to B=\Spec(R)$ is a degree-$2$ ramified cover. Note that $R_1$ is a DVR having quotient field
$K_1$ and residue field $k=\ov k$.
Consider the base change
\[f_1:\X_1:=\X\times_B B_1\to B_1,\]
and let $\sigma_1:B_1\to \X_1$ be the section of $f_1$ obtained by pulling back the section $\sigma$ of $f$.
The special fiber of $\X_1$ is isomorphic to $X$ and the total space  $\X_1$ has a singularity formally equivalent to $xy=u^2$ at each of the nodes
of the special fiber. It is well-known that the relatively minimal regular model  of $f_1:\X_1 \to B_1$, call it $\wh{f}:\wh{\X}\to B_1$,
is obtained by blowing-up $\X_1$ once at each one of these singularities. Moreover, the section $\sigma_1$ of $f_1$ admits a lifting to a section $\wh{\sigma}$ of $\wh{f}$
since the image of $\sigma_1$ is contained in the smooth locus of $\X_1$.
It is easy to check that the general fiber of $\wh{f}$ is equal to $\wh{\X}_{K_1}=\X_K\times _K K_1$ while its special fiber is equal to $\wh{\X}_k=\hX$.
In other words, $\wh{f}:\wh{\X}\to B_1$ is a one-parameter regular local smoothing of $\hX$.
By construction, it follows that we have a commutative diagram as in \eqref{compa-smooth} which, moreover, is cartesian on the general fibers of $f$ and $\wh f$.

\end{nota}

\begin{thm}
\label{quot} In the set up of \ref{two}, let $\un q$ be a polarization on $X$ of total degree $d=|\un q|$ and let $\un{\wh q}$ be the associated polarization on $\hX$ (see
\ref{pola-nota}).
Then there is a surjective $B_1$-morphism
\[
\tau_{\wh f}:J^{\hat\sigma}_{\hat f}(\hat{\un q})_{\rm sm}\cong N(\Pic^d\widehat\X_{K_1}) \la J^\sigma_f(\un q)\times_BB_1,
\]
which is an isomorphism over the general point of $B_1$.
\end{thm}


\begin{proof}
Let $P:=\widehat\sigma(k_1)\in \wh X_{\rm sm}$ and denote by     $v_0$ the vertex of the dual graph $\Gamma_{\wh X}$ of $\wh X$ corresponding to the irreducible component of $\wh X$ containing $P$. The fact that $J^{\hat\sigma}_{\hat f}(\hat{\un q})_{\rm sm}\cong N(\Pic^d\widehat\X_{K_1})$ is an immediate consequence of Theorem \ref{picner}. By (\ref{pfclaim}), we have
\[
J^{\wh\sigma}_{\wh f}(\wh{\un q})_{\rm sm} \cong \frac{\coprod _{\un d \in B_{\Gamma_{\wh X}}^{v_0}(\wh{\un q})} \Pic_{\wh f}^{\un d}}{\sim_{K_1}},
\]
where $\sim_{K_1}$ denotes the gluing along the general fibers of $\Pic_{\wh f}^{\un d}$ which are isomorphic to $\Pic^d({\wh X}_{K_1})$, where $d=|\wh{\un q}|=|\un q|$.
We will start by showing the existence of a $B_1$-morphism
\[\tau_{\wh f}^{\un d}: \Pic_{\wh f}^{\un d}\la J^\sigma_f(\un q)\times_BB_1\]
for every $\un d \in B_{\Gamma_{\wh X}}^{v_0}(\wh{\un q})$.
By the universal property of fiber products, the existence of $\tau_{\wh f}^{\un d}$ is equivalent
the existence of a morphism $\mu_{\wh f}^{\un d}:\Pic_{\wh f}^{\un d}\to J_f^\sigma(\un q)$ making the following diagram commute
\begin{equation*}
\xymatrix{
{\Pic_{\wh f}^{\un d}}\ar@/_/[ddr] \ar@/^/[drr]^{\mu_{\wh f}^{\un d}}\\
& {J_f^\sigma(\un q)\times_BB_1} \ar[r] \ar[d] & J_f^\sigma(\un q) \ar[d]\\
& B_1 \ar[r] & B
}
\end{equation*}

Now, since $J_f^\sigma(\un q)$ is a fine moduli space, such a morphism $\mu_{\wh f}^{\un d}$ is uniquely determined by a family of $(1,\sigma)$-quasistable torsion-free sheaves on $\Pic_{\wh f}^{\un d}\times_B\X$ with respect to $\un q$ (since all the singular fibers of $\Pic_{\wh f}^{\un d}\times_B\X\to \Pic_{\wh f}^{\un d}$ are isomorphic to $X$, we are slightly abusing the notation here: $\Pic_{\wh f}^{\un d}$ may very well not be a DVR): we fix our notation according to the following commutative diagram, where both the outward and the left inward diagrams are cartesian and the morphism $\hat \pi$ is the morphism induced by the inner commutativity of the diagram on the fiber product $\Pic_{\hat f}^{\un d}\times_B\X$.
\begin{equation*}
\xymatrix{
\Pic_{\wh f}^{\un d}\times_B\X \ar[rrrd] \ar[ddr]^{\tilde f}\\
&\Pic_{\wh f}^{\un d}\times_{B_1}\wh{\X}_1 \ar[ul]_{\wh \pi} \ar[r] \ar[d]_{\bar f}& \wh{\X}_1 \ar[r] \ar[d]_{\wh f} & \X \ar[d]_{f} \\
&\Pic_{\wh f}^{\un d} \ar[r] \ar@/_/[u]_{(1,\wh\sigma)} \ar@/^/[uul]^{(1,\sigma)} & B_1 \ar[r] \ar@/_/[u]_{\wh \sigma}& B \ar@/_/[u]_{\sigma}
}
\end{equation*}

The morphism $\wh \pi$ is then a $B$-morphism that is an isomorphism over the general point of $B$ while over the closed point of $B$ consists of blowing down all the exceptional components of the morphism $\pi:\wh X\to X$.
Since $\wh f$ is a family of projective curves with reduced and connected fibers having geometrically integral irreducible components and admitting a section $\wh \sigma$, it follows from the work of Mumford in \cite{mumford} that the relative Picard functor of $\wh f$ is representable (see \cite{fga}, Theorems 9.2.5 and 9.4.18.1).
Therefore, there exists a Poincar\'e sheaf $\mathcal P$ on $\Pic_{\wh f}^{\un d}\times_{B_1}\wh\X_1$ (see \cite{fga}, Exercise 9.4.3), i.e. a sheaf whose restriction to a fiber of $\bar f$ at a point $[C,L]$ of $\Pic_{\wh f}^{\un d}$ is isomorphic to $L$. The above description of $\wh \pi$ together with Theorem \ref{quot-blow}\eqref{quot-blow3} implies that $\I:=\wh\pi_*(\mathcal P)$ is a
family of $(1,\sigma)$-quasistable torsion-free sheaves with respect to $\un q$ over the family $\tilde f$. This yields uniquely a morphism $\mu_{\bar f}^{\un d}$ as already observed.

By construction, over the general point $\Spec K_1$ of $B_1$, the morphism $\tau_{\wh f}^{\un d}$ restricts to the natural isomorphism
\[ \Pic^d(\wh \X_{K_1})=\Pic^d(\X_K\times_K K_1)\stackrel{\cong}{\longrightarrow} \Pic^d(\X_K)\times_K K_1.\]
Therefore,  as $\un d$ varies on $B^{v_0}_{\Gamma_{\wh X}}(\wh{\un q})$, we can glue the morphisms $\tau_{\wh f}^{\un d}$ along the general fiber
to obtain the desired $B_1$-morphism $\tau_{\wh f}$. By construction the $B_1$-morphism $\tau_{\wh f}$ is an isomorphism over the general point of $B_1$
and, by Theorem \ref{quot-blow}\eqref{quot-blow3}, it is surjective over the closed point of $B_1$. This concludes the proof of the statement.


\end{proof}

\section{Comparing fine and coarse compactified Jacobians}

In this section, we investigate when a fine compactified Jacobian is isomorphic to its coarse compactified Jacobian.
Indeed, it turns out that the sufficient condition given by Esteves in \cite[Thm. 4.4]{est2} is also necessary (for nodal curves).

Throughout the whole section we will use the terminology introduced in paragraph  \ref{pola-nota} above.

\begin{thm}\label{nondeg-thm}
Let $X$ be a nodal curve and $\un q$ a polarization on $X$.
The following conditions are equivalent
\begin{enumerate}[(i)]
\item The polarization $\un q$ is non-degenerate;
\item For every $P\in \Xsm$ the map $\Phi:J_X^P(\un q)\to U_X(\un q)$ is an isomorphism;
\item There exists a point $P\in \Xsm$ such that the map $\Phi:J_X^P(\un q)\to U_X(\un q)$ is an isomorphism;
\item The number of irreducible components of $U_X(\un q)$ is equal to the complexity $c(\Gamma_X)$
of the dual graph $\Gamma_X$ of $X$.
\end{enumerate}
\end{thm}
\begin{proof}
The implication (i) $\Rightarrow$ (ii) follows from \cite[Thm. 4.4]{est2}. In fact, note that, although the theorem of loc. cit.
is stated in a weaker form, namely assuming the stronger hypothesis that $\displaystyle \un q_Y-\frac{\delta_Y}{2} \not \in \Z$ for all
subcurves $Y\subsetneq X$ which are not spines, a closer look at its proof reveals that the theorem holds under the weaker
hypothesis that $\un q$ is not integral at all the subcurves $Y\subsetneq X$ which are not spines.

(ii) $\Rightarrow$ (iii) is clear.

(iii) $\Rightarrow$ (iv) follows from the fact that the number of irreducible components of $J_X^P(\un q)$ is equal to
$c(\Gamma_X)$. Indeed, according to Theorem \ref{strata-Jac}, the number of irreducible components of $J_X^P(\un q)$ is equal to the number
of irreducible components of $\JXsm$, which, according to Proposition \ref{multdeg-sh} applied
to the case $S=\emptyset$,
is equal to $c(\Gamma_X)$.

(iv) $\Rightarrow$ (i):
 Fix a one-parameter  regular local smoothing $f:\X\to B=\Spec(R)$ of $X$
(see \ref{notner}).
Such a one-parameter  smoothing determines a commutative diagram:
\begin{equation}\label{diag-sm}
\xymatrix{
& N_X^d \ar^s[d] &  \\
J^{ss}_X(\un q)_{{\rm sm}} \ar@{}[drr]|{\square} \ar@{->>}^t[ur]\ar^p[r]& U_X(\un q) &\\
&J^{ss}_X(\un q)_{{\rm sm}}^0 \ar@{_{(}->}^{j'}[ul]\ar_{p'}[r]& \UXsm \ar@{{(}->}[ul]^{j}\ar_u[uul]\\
}
\end{equation}
that we now explain. $N_X^d:=N(\Pic^d\X_K)_k$ is the special fiber of the N\'eron model of $\Pic^d(\X_K)$ relative to $f$,
where $d:=|\un q|$.
$U_X(\un q)_{{\rm sm}}$ denotes the smooth locus of $U_X(\un q)$
and $j$ is its open immersion into $U_X(\un q)$.
$J^{ss}_X(\un q)_{{\rm sm}}$ denotes the variety parametrizing line bundles on $X$ that are $\un q$-semistable and $p$ is the natural map
sending a $\un q$-semistable line bundle into its class in $U_X(\un q)$, or in other words $p$ is induced by the universal
family of $\un q$-semistable line bundles over $J^{ss}_X(\un q)_{{\rm sm}}\times X$. $J^{ss}_X(\un q)_{{\rm sm}}^0$ is, by definition,  equal to
\[J^{ss}_X(\un q)_{{\rm sm}}^0:=\UXsm\times_{U_X(\un q)} J^{ss}_X(\un q)_{{\rm sm}},\]
and $j', p'$ are the induced maps.
The maps $t$ and $u$ are the special fibers of two maps over $B$ induced by the N\'eron mapping property: indeed $J^{ss}_X(\un q)_{{\rm sm}}$ (resp. $\UXsm$)
is the special fiber of a $B$-scheme $\Pic_f^{ss}$ (resp. $U_f(\un q)_{\rm sm}$) smooth over $B$ whose generic fiber is $\Pic^{d}(\X_K)$.
Note also that the map $t$ is the restriction to $J^{ss}_X(\un q)_{{\rm sm}}\subset \Pic^d(X)$ of the special fiber of the map $q:\Pic_f^d\to N(\Pic^d(\X_K))$ (see (\ref{sepquot})).
From the explicit description of the map $q$ given in \ref{notner} and the fact that
every element in the degree class group $\Delta^d_X$ of $X$ can be represented by
a $\un q$-semistable line bundle on $X$ (as it follows from Proposition \ref{card-qs}),
we deduce that $t$ is surjective.
Finally, the map $s$ is induced by the fact that $U_f(\un q)$ is separated over $B$ and $N(\Pic^d(\X_K))$ is the biggest separated quotient of
the non-separated $B$-scheme $\Pic_f^{ss}$ (see \ref{notner}).

$\underline{\text{Claim 1}}$: $p'$ is surjective.

Consider a polystable sheaf $\I\in \UXsm$. According to Fact \ref{smooth}(\ref{smooth2}), the set of nodes $\NF(\I)$ at which $\I$ is not free is contained in
$\Xsep$.
The surjectivity of $p'$ is equivalent to showing that
there exists a $\un q$-semistable line bundle $L$ in the same $S$-equivalence class of $\I$.
By decreasing induction on the cardinality of $\NF(\I)$, it is enough to show that given $n\in \NF(\I)$ there exists
$\I'\in \UXsm$ such that $\I'$ is $S$-equivalent to $\I$ and $\NF(\I')=\NF(\I)\setminus \{n\}$.
Let $T_1$ and $T_2$ be the tails attached to $n$, and set $I_i:=I_{T_i}$. Since $n$ is a separating node, it follows from \cite[Example 38]{est1}
that $\I=\I_1\oplus \I_2$. To conclude, it is enough to take a non-trivial extension
\[0\to \I_1\to \I'\to \I_2\to 0,\]
whose existence follows from \cite[Lemma 4]{est1}.

$\underline{\text{Claim 2}}$: If $u$ is surjective then $\Im p \subseteq \UXsm$.

If $u$ is surjective then, using that $p'$ is surjective by the Claim 1, we get that $t \circ j'=u\circ p'$ is surjective.
From the diagram (\ref{diag-sm}) we easily get that $\Im (s\circ t \circ j')\subseteq \UXsm$.
This, together with the surjectivity of $t\circ j'$ implies that $\Im s \subseteq \UXsm$. Since
$\Im p\subseteq \Im s$ because $t$ is surjective, we get the conclusion.

Let us now conclude the proof of the implication (iv) $\Rightarrow$ (i). Assume that the number of irreducible components of
$U_X(\un q)$ is equal to $c(\Gamma_X)$.
This means that $u$ is surjective (and hence an isomorphism). By Claim 2, we deduce that $\Im p \subseteq \UXsm $.
We claim that this implies that $\un q$ is non-degenerate. Indeed, if this were not   the case then, by Lemma \ref{exis-linebun} below, there would
exist a $\un q$-semistable line bundle $L$ such that $\displaystyle \deg_ZL=\un q_Z-\frac{\delta_Z}{2}$
for some proper subcurve $Z\subsetneq X$ which is not a spine. But then clearly $Z\cap Z^c\subset \NF(\Gr(L))\not \subset \Xsep$ which would
imply that  $p(L)=[\Gr(L)]\not\in \UXsm$ by Fact \ref{smooth}(\ref{smooth2}).
\end{proof}

\begin{lemma}\label{exis-linebun}
If a polarization $\un q$ on $X$ is not general  then there exists a subcurve $Z\subsetneq X$ with both $Z$ and $ Z^c$ connected
and a $\un q$-semistable line bundle $L$ on $X$ such that $\displaystyle \deg_Z L=\un q_Z-\frac{\delta_Z}{2}$.
Moreover, if $\un q$ is not non-degenerate, then we can choose $Z$ not to be a spine.
\end{lemma}
\begin{proof}
By assumption, $\un q$ is integral at a proper subcurve $Y\subsetneq X$.
Chose a connected component of $Y$ and call it $Z'$. Set $Z$ to be one of the connected components of $Z'^c$. Clearly $Z$ and $Z^c$ are connected.

If moreover $\un q$ is not non-degenerate then there exists a subcurve $Y\subsetneq X$ as before which, moreover, is not a spine.
Then we can chose a subcurve $Z'$ as before in such a way that is it not a spine. This easily implies that $Z$ is
not a spine as well.

From the assumption that $\un q$ is integral at $Y$ and from the construction of $Z$, we deduce  that
$\displaystyle \un q_Z-\frac{\delta_Z}{2}\in \Z$ and that
 $\displaystyle \un q_{Z^c}-\frac{\delta_{Z^c}}{2}=|\un q|-\un q_Z-\frac{\delta_Z}{2}\in \Z$.

Consider the restriction $\un q_{|Z}$ of the polarization $\un q$ at $Z$ (see \ref{pola-nota}).
Since $Z$ is connected, the complexity of its dual graph $\Gamma_Z$ is at least one and therefore
Proposition \ref{multdeg-sh} implies that, for any chosen smooth point $P\in \Xsm$,  there exists a line bundle
$L_1$ on $Z$ that is $\un q_{|Z}$-P-quasistable, and in particular $\un q_{|Z}$-semistable. This means
that  for any subcurve $W_1\subset Z$ it holds:
\begin{equation}\label{bal_1}
\begin{sis}
&\deg_{Z} L_1=|\un q_{|Z}|=\un q_Z-\frac{\delta_Z}{2}, \\
&\deg_{W_1} L_1 \geq (\un q_{|Z})_{W_1}-\frac{|{W_1}\cap \ov{Z\setminus {W_1}}|}{2}=\un q_{W_1}-\frac{|{W_1}\cap Z^c|}{2} -\frac{|{W_1}\cap \ov{Z\setminus {W_1}}|}{2}=\\
&=\un q_{W_1}-\frac{\delta_{W_1}}{2}.
\end{sis}
\end{equation}
Analogously, consider the polarization $\w{\un q}$ on $Z^c$ given by
\[\w{\un q}_R:=\un q_R+\frac{|R\cap Z|}{2} \text{ for any subcurve } R\subset Z^c.\]
Since $Z^c$ is connected, there exists a line bundle $L_2$ on $Z^c$ that is $\w{\un q}$-semistable, i.e. such that
for any subcurve ${W_2}\subset Z^c$ it holds:
\begin{equation}\label{bal_2}
\begin{sis}
&\deg_{Z^c} L_2=|\w{\un q}|=\un q_{Z^c}+\frac{\delta_{Z^c}}{2}, \\
&\deg_{W_2} L_2 \geq \w{\un q}_{W_2}-\frac{|{W_2}\cap \ov{Z^c\setminus {W_2}}|}{2}=\un q_{W_2}+\frac{|{W_2}\cap Z|}{2} -\frac{|{W_2}\cap \ov{Z^c\setminus {W_2}}|}{2}= \\
&=\un q_{W_2}-\frac{\delta_{W_2}}{2}+|{W_2}\cap Z|.
\end{sis}
\end{equation}
Now let $L$ be a line bundle on $X$ such that $L_{Z}=L_{| Z}=L_1$ and $L_{Z^c}=L_{|Z^c}=L_2$ (obviously such an
$L$ exists). Using equations (\ref{bal_1}) and (\ref{bal_2}), we have that
\begin{equation}\label{totbal1}
\deg L=\deg_Z L_1+\deg_{Z^c} L_2=\un q_Z-\frac{\delta_Z}{2}+\un q_{Z^c}+\frac{\delta_{Z^c}}{2}=|\un q|.
\end{equation}
For any subcurve $W\subset X$, let $W=W_1\cup W_2$ where $W_1:=W\cap Z$ and $W_2:=W\cap Z^c$.
Using equations (\ref{bal_1}) and (\ref{bal_2}), we compute
\begin{equation}\label{totbal2}
\deg_W L=\deg_{W_1}L_1+\deg_{W_2}L_2\geq \un q_{W_1}-\frac{\delta_{W_1}}{2} +\un q_{W_2}-\frac{\delta_{W_2}}{2}+|W_2\cap Z|
\geq
 \end{equation}
\[\geq \un q_W-\frac{\delta_{W_1}}{2}-\frac{\delta_{W_2}}{2}+|W_1\cap W_2|=\un q_W-\frac{\delta_W}{2}.\]
The above equations (\ref{totbal1}) and (\ref{totbal2}) says that $L$ is $\un q$-semistable.
On the other hand,  from equation (\ref{bal_1}) we get $\displaystyle \deg_Z L=\un q_Z-\frac{\delta_Z}{2}$.

\end{proof}

\begin{nota}{\emph{Relation between non-degenerate and general polarizations}}

The aim of this subsection is to discuss the relation between a polarization $\un q$ being non-degenerate and
the stronger condition of being general (see Def. \ref{def-pol}).
We begin by describing the geometric meaning of being general.

\begin{prop}\label{gen-teo}
The following conditions are equivalent
\begin{enumerate}[(i)]
\item $\un q$ is general (see Def. \ref{def-pol}(\ref{def-pol1}));
\item Every $\un q$-semistable sheaf is $\un q$-stable, i.e. $U_X^s(\un q)=U_X(\un q)$;
\item Every $\un q$-semistable simple sheaf is $\un q$-stable, i.e. $J_X^s(\un q)=J_X^{ss}(\un q)$;
\item Every $\un q$-semistable line bundle is $\un q$-stable.
\end{enumerate}
\end{prop}
\begin{proof}
(i) $\Rightarrow$ (ii): If $\un q$ is general then the right hand side of the inequality (\ref{multdeg-sh1}) is never an integer.
Hence the inequality in (\ref{multdeg-sh1}), if satisfied, is always strict, from which the conclusion follows.

The implications (ii) $\Rightarrow$ (iii) $\Rightarrow $ (iv) are clear.

(iv) $\Rightarrow$ (i): If $\un q$ is not general, then Lemma \ref{exis-linebun} implies that there exists a $\un q$-semistable line bundle $L$ on $X$
that is not $\un q$-stable.

\end{proof}

\begin{remark}
The implication (i) $\Rightarrow$ (iii) was proved   in \cite[Prop. 3.5]{est2}.
\end{remark}

\begin{remark}\label{can-gen}
The canonical polarization of degree $d$ on $X$ of Rmk. \ref{inequ}(\ref{inequ2}) is general if and only $X$ is
$d$-general in the sense of \cite[Cor.-Def. 4.13]{capneron} (see also \cite[Def. 1.13]{capNtype}), as it follows easily
by comparing the definition of loc. cit. with the above Proposition \ref{gen-teo}.
\end{remark}

In the remaining of this subsection, we want to give an answer to the following

\begin{question}
How far is a non-degenerate polarization from being general?
\end{question}

Denote by $X^2$ any smoothing of $X$ at the set of separating nodes $\Xsep$ of $X$.
Given a subcurve $Z\subset X^2$, denote by     $\ov{Z}$ the subcurve  of $X$ to which
$Z$ specializes. Observe that $g_{\ov Z}=g_Z$ and $\delta_{\ov Z}=\delta_Z$.
A subcurve $Y\subset X$ is of the form $Y=\ov{Z}$ for some subcurve $Z\subset X^2$
if and only if
\begin{equation}\label{subcur}
Y\cap Y^c\cap \Xsep=\emptyset.
\end{equation}

Given a polarization $\un q$ on $X$, we define a polarization $\un q^2$ on any
smoothing $X^2$ by $\un q^2_Z:=\un q_{\ov Z}$ for any subcurve $Z\subset X^2$.
Observe that, although the smoothing $X^2$ is not unique, its combinatorial type (i.e. its weighted dual
graph) and the polarization $\un q^2$ are uniquely determined.

\begin{prop}\label{nondeg-gen}
A polarization $\un q$ on $X$ is non-degenerate if and only if, for every (or equivalently, for some)
smoothing $X^2$ of $X$ at its set of
separating nodes, the induced polarization $\un q^2$ on $X^2$ is general.
\end{prop}
\begin{proof}
Assume that $\un q$ is non-degenerate on $X$. Let $Z$ be a proper subcurve of any fixed smoothing $X^2$
and $W$ a connected component of $Z$ or $Z^c$. We want to show that $\displaystyle \un q^2_W-\frac{\delta_W}{2}\not\in \Z$.
Consider the subcurve $\ov{Z}\subset X$. Clearly $\ov{Z}$ is a proper subcurve and is not a spine because
of (\ref{subcur}). Moreover $\ov{W}$ is a connected component of $\ov{Z}$ or $\ov{Z}^c$.
Therefore, because of the assumption and the definition of $\un q^2$, we get
$\displaystyle \un q^2_W-\frac{\delta_W}{2}=\un q_{\ov W} -\frac{\delta_{\ov W}}{2}\not\in \Z$.

Conversely, assume that $\un q^2$ is general for some fixed smoothing $X^2$ and,
by contradiction, assume also that $\un q$ is not non-degenerate on $X$. Then there
exists some subcurve $Y$ of $X$ such that
\begin{equation}\label{prop-sub}
\begin{sis}
& Y \text{ is connected, } \\
& Y\cap Y^c\not \subset \Xsep \text{ (i.e. $Y$ is not a spine}), \\
& \un q \text{ is integral at } Y.
\end{sis}
\end{equation}
If we chose $Y$ maximal among the subcurves satisfying the properties (\ref{prop-sub}), then
we claim that $Y\cap Y^c\cap \Xsep=\emptyset.$
Indeed, if this is not the case, then there exists a separating node $n\in Y\cap Y^c$.
Since $Y$ is connected, one of the two tails attached to $n$, call it $T$, is a connected
component of $Y^c$. Consider the subcurve $Y':=Y\cup T$. It is easily checked that
$Y'$ is connected, $Y'\cap Y'^c=(Y\cap Y^c)\setminus \{n\}\not\subset \Xsep$
and that $\un q$ is integral at $Y'$. Therefore $Y'$ satisfies the properties
(\ref{prop-sub}) and, since $Y\subsetneq Y'$, this contradicts the maximality of $Y$.

Since the chosen maximal subcurve $Y$ satisfies property (\ref{subcur}), we known
that there exists a subcurve $Z\subsetneq X^2$ such that $\ov{Z}=Y$. But then
the same  argument as before gives that $\un q^2$ is integral
at $Z$, which contradicts the initial assumption on $\un q^2$.

\end{proof}

\begin{remark}\label{can-nondeg}
The canonical polarization of degree $d$ on $X$ of Rmk. \ref{inequ}(\ref{inequ2}) is non-degenerate if and only $X$ is
weakly $d$-general in the sense of \cite[Def. 1.13]{capNtype}, as it follows easily by comparing the definition of loc. cit. with the above
Proposition \ref{nondeg-gen}. Using this, the equivalence $(i)\Leftrightarrow (iv)$ of our Theorem \ref{nondeg-thm} recovers \cite[Thm. 2.9]{capNtype}
in the case of the canonical polarization of degree $d$.
\end{remark}

\end{nota}

\end{document}